\documentclass[11pt,reqno]{amsart}

\usepackage[margin=1.5in]{geometry}

\usepackage{style}

\let\cong\equiv\relax
\renewcommand{\equiv}{\simeq}
\renewcommand{\isom}{\equiv}

\DeclareMathName{\Fun}{Fun}
\newcommand{\FunAcc}{\Fun^\mathrm{acc}}
\newcommand{\FunR}{\Fun^\rmR}
\newcommand{\FunL}{\Fun^\rmL}
\DeclareMathName{\PSh}{PSh}
\DeclareMathName{\Sh}{Sh}
\newcommand{\PShAcc}{\PSh^\mathrm{acc}}
\DeclareMathName{\Sp}{Sp}
\DeclareMathName{\CMon}{CMon}
\DeclareMathName{\CAlg}{CAlg}
\DeclareMathName{\cCAlg}{cCAlg}
\DeclareMathName{\DAlg}{DAlg}
\DeclareMathName{\Mod}{Mod}
\DeclareMathName{\LMod}{LMod}
\DeclareMathName{\RMod}{RMod}
\DeclareMathName{\QCoh}{QCoh}
\DeclareMathName{\Perf}{Perf}
\DeclareMathName{\PStk}{PStk}
\DeclareMathName{\Stk}{Stk}
\DeclareMathName{\Sch}{Sch}
\DeclareMathName{\AffSch}{AffSch}
\DeclareMathName{\Top}{Top}
\DeclareMathName{\CGHaus}{CGHaus}
\DeclareMathName{\Cond}{Cond}
\DeclareMathName{\Ani}{Ani}
\DeclareMathName{\Set}{Set}
\DeclareMathName{\Prof}{Prof}
\DeclareMathName{\Lat}{Lat}
\newcommand{\aCAlg}{\operatorname{\fraka\CAlg}}
\newcommand{\CAlgwl}{\CAlg^\mathrm{wl}}

\DeclareMathOperator{\Tot}{Tot}
\DeclareMathOperator*{\flim}{``lim''\,}
\DeclareMathOperator{\sk}{sk}

\newcommand{\LGamma}{{\rmL\Gamma}}

\newcommand{\PSigma}{\calP_\Sigma}

\newcommand{\Cat}{\cat{Cat}}

\newcommand{\Finp}{{\cat{Fin_*}}}

\newcommand{\Einfty}{{\bbE_\infty}}
\DeclareMathName{\LSym}{LSym}
\newcommand{\proj}{\mathrm{proj}}

\DeclareMathName{\GL}{GL}

\newcommand{\et}{\acute{\mathrm{e}}\mathrm{t}}
\newcommand{\pet}{{p\mathhyphen\et}}

\newcommand{\dR}{\mathrm{dR}}
\newcommand{\FdR}{\mathrm{FdR}}
\newcommand{\Hodge}{\mathrm{Hdg}}
\newcommand{\dRc}{\mathrm{dR,c}}

\newcommand{\qFdR}{{q\mathhyphen\FdR}}

\newcommand{\shtuka}{{\cyrsymbol{\cyrsh}}}

\newcommand{\Prism}{\mathbbl{\Delta}}
\newcommand{\Prismhat}{{\widehat{\Prism}}}

\newcommand{\Prismtilde}{{\widetilde{\Prism}}}

\newcommand{\Nyg}{\calN}
\newcommand{\Nyghat}{{\widehat{\Nyg}}}

\newcommand{\NNyg}{\mathcal{NN}}

\newcommand{\Syn}{\mathrm{Syn}}

\newcommand{\ZpPrism}{\Zp^\Prism}
\newcommand{\ZpN}{\Zp^\Nyg}
\newcommand{\ZpSyn}{\Zp^\Syn}

\newcommand{\jdR}{j_\mathrm{dR}}
\newcommand{\jHT}{j_\mathrm{HT}}

\newcommand{\ZpNN}{\Zp^\NNyg}

\newcommand{\ZpqFdR}{\Zp^\qFdR}

\newcommand{\QFdR}{\bbQ^\FdR}

\newcommand{\Mfg}{\calM_\mathrm{fg}}
\newcommand{\Guniv}{\widehat{\bbG}_\mathrm{univ}}

\DeclareMathName{\THH}{THH}
\DeclareMathName{\TC}{TC}
\newcommand{\TCm}{\TC^-}
\DeclareMathName{\TP}{TP}

\DeclareMathName{\TR}{TR}
\newcommand{\TCmref}{\TC^{-,\mathrm{ref}}}

\DeclareMathName{\RingStk}{RingStk}
\DeclareMathName{\AlgStkplain}{AlgStk}
\newcommand*{\AlgStk}[1]{{#1\mathhyphen\AlgStkplain}}
\DeclareMathName{\AbMonStk}{AbMonStk}

\newcommand{\sharphat}{\widehat{\sharp}}

\newcommand{\Ga}[1][]{\bbG_{\mathrm{a}\ifthenelse{\equal{#1}{}}{}{,#1}}}
\newcommand{\Gm}[1][]{\bbG_{\mathrm{m}\ifthenelse{\equal{#1}{}}{}{,#1}}}

\newcommand{\BGm}{\bfB\Gm}
\newcommand{\Gasharp}{\Ga^\sharp}

\newcommand{\Gahat}{\widehat{\bbG}_\mathrm{a}}
\newcommand{\Gmhat}{\widehat{\bbG}_\mathrm{m}}

\newcommand{\Gasharphat}{\Ga^{\sharphat}}

\NewDocumentCommand{\Gmperf}{O{} O{}}{\Gm[#1]^{\perf\ifthenelse{\equal{#2}{}}{}{,#2}}}
\newcommand{\BGmperf}{\bfB\Gmperf}

\newcommand{\Mm}[1][]{\bbM_{\mathrm{m}\ifthenelse{\equal{#1}{}}{}{,#1}}}
\NewDocumentCommand{\Mmperf}{O{} O{}}{\Mm[#1]^{\perf\ifthenelse{\equal{#2}{}}{}{,#2}}}
\newcommand{\BMm}{\bfB\Mm}
\newcommand{\BMmperf}{\bfB\Mmperf}

\newcommand{\Gaflat}{\Ga^\flat}
\newcommand{\Gabar}[1][]{\overline{\bbG}_{\mathrm{a}\ifthenelse{\equal{#1}{}}{}{,#1}}}

\newcommand{\Fp}{\bbF_p}
\newcommand{\Fq}{\bbF_q}
\newcommand{\Zp}{\bbZ_p}

\newcommand{\Zpzetap}{{\Zp[\zeta_p]}}

\newcommand{\Fpbar}{\overline{\bbF}_p}

\newcommand{\Ainf}{A_\mathrm{inf}}

\newcommand{\QdR}{{\Zp\Brk{q-1}}}
\newcommand{\QdRtw}{{\Zp\Brk{q^{1/p}-1}}}

\newcommand{\QdRNyg}{{(\Zpzetap / \QdR)^\Nyg}}
\newcommand{\QdRNyghat}{{(\Zpzetap / \QdR)^\Nyghat}}

\newcommand{\ZpcycNyg}{\Zp^{\mathrm{cyc},\Nyg}}

\newcommand{\perf}{\mathrm{perf}}

\newcommand{\Aohat}{{\widehat{\bbA}^1}}
\newcommand{\Aoperf}{\bbA^{1,\perf}}
\newcommand{\AomodGm}{{\bbA^1 \!/\, \Gm}}
\newcommand{\AohatmodGm}{{\Aohat \!/\, \Gm}}
\newcommand{\AomodMm}{{\bbA^1 \!/\, \Mm}}

\newcommand{\Wperf}[1][]{{W^{\perf\ifthenelse{\equal{#1}{}}{}{,#1}}}}
\newcommand{\Wrat}[1][]{{W^{\mathrm{rat}\ifthenelse{\equal{#1}{}}{}{,#1}}}}
\newcommand{\Wratplus}{{\Wrat[+]}}
\newcommand{\Wbig}[1][]{W^{\mathrm{big}\ifthenelse{\equal{#1}{}}{}{,#1}}}
\newcommand{\Wbigperf}{{\Wbig[\perf]}}

\DeclareMathName{\MU}{MU}

\newcommand{\pdesc}{\mathrm{pdesc}}

\newcommand{\pinspec}[1]{\pi_0 \abs{\Spec #1}}
\newcommand{\abspec}[1]{\abs{\Spec #1}}

\newcommand{\Ntr}{\mathbb{N}_\mathrm{tr}}

\newcommand{\slashsub}[1]{\mathbin{/_{\!\!#1}}}

\title{Cohomology theories in the moduli of ring stacks}

\author{Dhilan Lahoti}
\address{%
  Department of Mathematics \\%
  Harvard University \\%
  Cambridge, Massachusetts, United States of America%
}
\email{dhilanlahoti@gmail.com}
\urladdr{https://dhilanlahoti.com/}

\author{Deven Manam}
\address{%
  Department of Mathematics \\%
  Northwestern University \\%
  Evanston, Illinois, United States of America%
}
\curraddr{%
  Department of Mathematical Sciences \\%
  University of Copenhagen \\%
  Copenhagen, Denmark%
}
\email{deven.manam@gmail.com}
\urladdr{https://manam.dev/}

\begin{document}

\vspace{-1.75em}
\begin{abstract}
We show that the natural map from the syntomification of a ring $R$ to the stack of $R$-algebra stacks is fully faithful, answering a question of Drinfeld, and we describe its essential image in terms of underlying monoid stacks.
We also give similar statements in the characteristic $0$ filtered de Rham, $\ell = p$ étale, and Betti settings.
\end{abstract}

\maketitle

\tableofcontents

\section*{Introduction}

The ``stacky'' approach to the cohomology of algebraic varieties, first introduced by Simpson \cite{simpsondRstack}, has recently attracted substantial interest in the $p$-adic setting due to work of Drinfeld \cite{stackycrystals,prismatization} and Bhatt--Lurie \cite{apc,blprismatization}.
A salient feature of this approach is the central role played by \emph{ring stacks}, which provide a convenient way of packaging a cohomology theory, together with its coefficients, into a single object.
The purpose of this paper is to see how much mileage one can get by treating ring stacks themselves as the central objects of study.

Below is our main theorem in the syntomic setting.

\begin{theorem}[\Cref{ZpSynmainthm}, {\cite[Question 8.3.6]{prismatization}}]
\label{mainthmintro}
For a ring $A$, the map
\[ A^\Syn \to \AlgStk{A} \]
from the syntomification of $A$ to the stack of $A$-algebra stacks (classifying $(\Ga[A])^\Syn$) is fully faithful.\footnotemark
\footnotetext{Note that $(\Ga[A])^\Syn$ is by definition an $A$-algebra stack over $A^\Syn$, so, explicitly, given a point $\Spec S \to A^\Syn$, one obtains a map $\Spec S \to \AlgStk{A}$ classifying the pullback of $(\Ga[A])^\Syn$ to $S$.}
An $A$-algebra stack is in the essential image of this map if and only if its underlying abelian monoid stack (given by its multiplicative monoid) is in the image of the composite
\[ A^\Syn \to \AlgStk{A} \to \AbMonStk \mpunc. \]
\todo*[pointwise criterion?]%
\end{theorem}

Following \cite{motivesandringstacks}, we view the full faithfulness statement above as evidence that prismatic $F$-gauges fully capture some portion of the theory of motives.

The appearance of the underlying monoid stack above may seem somewhat strange.
The authors were motivated to consider it by the fact that the underlying multiplicative monoid structure of $\Ga^\Syn$ -- unlike the full ring structure -- appears naturally over a deeper base than $\ZpSyn$ (see \Cref{rmk:qFdR} and \Cref{evenstacks}).

From another perspective, one can view the description of the essential image above as saying that the theory of syntomification is ``defined over $\bbF_1$'':\footnotemark{} just as $A^\Syn$ is by definition the stack of $A$-algebra structures on the $\Zp$-algebra stack $\Ga^\Syn$, $\ZpSyn$ is the stack of $\Zp$-algebra structures on a certain ``$\bbF_1$-algebra'' (i.e.\ monoid) stack.
\footnotetext{This should not be confused with the existence of the stack $\bbF_1^{\fakeheight{\Syn}{S}}$ of \cite{f1syn}, whose name is somewhat misleading from this perspective.}

\begin{remark}
\Cref{mainthmintro} suggests the following strategy for constructing cohomology theories for schemes: one should first construct a suitable monoid stack (or a suitable substack of the stack of monoid stacks), then study the moduli of \emph{all} ring structures on it.
In other words, one constructs a cohomology theory for monoids -- a relatively small amount of data, though still nontrivial -- and considers all possible extensions to a cohomology theory for rings.
The theorem essentially states that taking this approach in the setting of prismatic or syntomic cohomology yields the expected answer.

In this direction, \Cref{mainthmintro} has motivated work of Scholze and Aoki on monoid stacks in the context of ball-invariant Berkovich motives \cite[Lecture XIII]{gestalten}.
In particular, Scholze conjectures that any ``reasonable'' monoid stack (more precisely, monoid Gestalt) in that setting locally admits a ring structure \cite[Conjecture 13.7]{gestalten}.
We note that this is in stark contrast to the (very non-ball-invariant) syntomic setting considered here, where there seem to be many reasonable monoid stacks which admit no ring structure, even locally (see \Cref{conj:ZpSynSSyn}).
\end{remark}

We also give analogues of \Cref{mainthmintro} in the settings of characteristic $0$ filtered de Rham cohomology (\Cref{QFdRmainthm}), $\ell = p$ étale cohomology (\Cref{petmainthm}), and Betti cohomology (\Cref{Bettimainthm}).
Each variant requires its own methods,\footnotemark{} which we briefly overview below.
\footnotetext{Except $\ell = p$ étale, which follows very easily from the methods of \Cref{sec:syntomic}.}

The syntomic case composes most of the paper.
The argument for full faithfulness is fairly straightforward: it amounts to the claim that, given a filtered Cartier--Witt divisor $M \to W$, one can obtain the induced map $\Wperf \to W/M$ purely from the ring stack $W/M$.
Identifying the essential image, however, is much more difficult.

For this, we first engage in a general study of affine $W$-module schemes in \Cref{sec:Wmod}, then develop the notions of \emph{passable $W$-modules} and \emph{polyfiltered Cartier--Witt divisors}, which are generalizations of admissible $W$-modules and filtered Cartier--Witt divisors, respectively, in \Cref{sec:passable}.
We expect polyfiltered Cartier--Witt divisors to be related to de Rham--Witt forms and $\TR$, in the same way that filtered Cartier--Witt divisors are related to the Nygaard filtration and $\TCm$ (see \Cref{rmk:vfCWTR}).
In \Cref{sec:coveringMmN,sec:monoidstorings,sec:WperftoW} we prove a number of results which allow us to bootstrap up from monoid maps out of $\Mm$ to ring maps out of $W$.\footnotemark
\footnotetext{We write $\Mm$ for the multiplicative monoid structure on $\bbA^1$, following \cite[Appendix B]{drinfeldfg}.}
Finally, in \Cref{sec:synmainresult} we assemble these ingredients to prove the main theorem.

The argument in the filtered de Rham case rests primarily on understanding maps between the group sheaves $\Gahat$ and $\Ga$.
From a sufficient understanding of these, we deduce that maps between the ring stacks in question are computed by maps between generalized Cartier divisors, which is exactly what is needed for full faithfulness.
Note that in this case we do not have a description of the essential image; see \Cref{rmk:QFdRmonoidstk}.

The Betti stack is of a somewhat different flavor from the others studied here, and the proof of full faithfulness in this case is correspondingly somewhat different from those in the other cases.
We begin by showing that, for suitable rings $R$, the Betti stack construction produces a fully faithful map from compactly generated Hausdorff spaces over the profinite set $\pinspec{R}$ to stacks over $R$, partially generalizing a result of Gregoric \cite[Theorem C]{condensedstoneduality}.
This statement implies that the maps of relevant ring stacks are computed by maps of families of topological rings over $\pinspec{R}$.
From this the desired full faithfulness follows easily, and the identification of the essential image reduces to point-set topology.

Finally, in \Cref{sec:complements} we suggest some directions for future work.
We begin by explaining why the naive derived analogue of our main theorem should not hold, and state an expectation as to how it can be repaired.
We then go on to state a number of other expectations and conjectures suggested by the theorem.
Among these are \Cref{conj:TCmref}, which describes Efimov's refined $\TCm$ of the rational numbers \cite{refinedTCm}, and \Cref{conj:f1syn}, which gives a moduli description of Lurie's stack $\bbF_1^\Syn$ \cite{f1syn}.

We also prove some results in \Cref{sec:LGamma} related to the derived divided powers functors which we expect to be of independent interest.
We show in particular that these functors preserve descendable morphisms of rings.\footnotemark
\footnotetext{In fact, our proof applies to any lax symmetric monoidal excisive functor, although there do not seem to be many natural examples of these besides the derived divided powers.}

\subsection*{Acknowledgments}
We thank Ben Antieau, Ko Aoki, Sanath Devalapurkar, Elden Elmanto, Lance Gurney, Ryomei Iwasa, Mitya Kubrak, Jacob Lurie, Akhil Mathew, Josh Mundinger, Arpon Raksit, Peter Scholze, Ferdinand Wagner, and Yuanning Zhang for many helpful conversations relevant to this work.
We are particularly grateful to Jacob Lurie for several fruitful suggestions, including the idea of reducing from $\Wbig$ to $\Mm$.
We also thank Peter Scholze for pointing out some mistakes in a previous version of this paper, and for suggesting the use of w-local rings in \Cref{sec:betti}.

The argument of \Cref{rmk:ZpSynff2} showing full faithfulness in \Cref{ZpSynmainthm} was independently discovered by Akhil Mathew.

This work was supported in part by the Simons Collaboration on Perfection.
DL was additionally supported by NSF grants DGE-2140743 and DMS-2305373, while DM was additionally supported by NSF grant DMS-2102010.
DM also thanks the Max Planck Institute for Mathematics and the University of Copenhagen for their hospitality.

\subsection*{Conventions}

All rings and monoids are by default commutative.

All of our categories are $\infty$-categories; specifically, we use the theory of quasicategories developed by Boardman--Vogt, Joyal, and Lurie.
However, our rings, schemes, etc.\ are by default classical.

For a morphism $Y \to X$ in a topos, we write $\RGamma((X, Y), -)$ for the ``relative sections''
\[ \fib(\RGamma(X, -) \to \RGamma(Y, -)) \mpunc. \]

We use $\Ga$ to refer to the scheme $\bbA^1$ with its additive group or ring structure,\footnotemark{} but we use $\Mm$ to refer to it with its multiplicative monoid structure, after \cite[Appendix B]{drinfeldfg}.
\footnotetext{
  Lars Hesselholt has pointed out to us that the ring structure on $\bbA^1$ is more properly referred to as $\calO$.
  However, $\calO$ tends to get somewhat overloaded with meanings and decorations already, so we use $\Ga$ in accordance with the usual convention in the subject.
}

``Complete'' is always taken to mean ``derived complete'' unless otherwise specified.

We write $W$ for the ring of $p$-typical Witt vectors and $\Wbig$ for the ring of big Witt vectors.
We also write $\Wrat(R)$ for the subring of $\Wbig(R) = 1 + t R \Brk{t}$ consisting of rational functions, and $\Wratplus \subseteq \Wrat$ for the subrig consisting of polynomials.
We will also make use of the inverse limit perfection of $W$, resp.\ $\Wbig$, denoted $\Wperf$, resp.\ $\Wbigperf$.

For a coaccessible category $\calC$, we use $\PShAcc(\calC)$ to denote the category of accessible presheaves on $\calC$, as introduced in our context by Waterhouse \cite{waterhousefpqc}; see \cite[\S A]{dirac2} for an overview in the setting of higher categories.

When we refer to sheaves on the category of algebras over an (animated) ring, we use the fpqc topology unless otherwise specified.
For us, a stack is generally an accessible such sheaf (although we also say ``stack of ring stacks'' etc.\footnotemark).
\footnotetext{It follows easily from \cite[Corollary 4.7.5.3]{ha} that the functor sending a ring to the category of ring stacks over it is a sheaf (\Cref{ringstksheaf}), but it is valued in large categories, and it is not accessible.}
We write $\Stk_R$ for the category of stacks over a ring $R$.

We use the term ``ring stack'' to refer to an accessible sheaf of animated rings.
Similarly, for an animated ring $R$, we say ``$R$-algebra stack'' to refer to an accessible sheaf of animated $R$-algebras.

For an animated ring $R$, we write $R_\red$ for the reduced locus of $R$, which by definition is the reduced locus of $\pi_0(R)$.

For a $p$-complete animated ring $R$, we write $R^\flat \defeq (R/p)^\perf$; that is, $R^\flat$ is the inverse limit perfection of $R/p$.
If $R$ is a sheaf of $p$-complete animated rings, we write $R^\flat$ for the pointwise application of $(-)^\flat$; note that by \Cref{Mmperfvsflat} this remains a sheaf.

Some stacks which we consider have natural enlargements to stacks of categories.
Following \cite{prismatization}, we refer to these enlargements as \emph{c-stacks}, and (when contrast is needed) to stacks of anima as \emph{g-stacks}.
However, unless immediately otherwise specified, we will always work only with underlying g-stacks.

We use ``commutative monoid'' etc.\ to refer to $\Einfty$-commutativity, while we say ``abelian monoid'' etc.\ to refer to what is sometimes called ``strict commutativity'' (see \Cref{abmon}).

\foralltheorems{\numberwithin{#1}{subsection}}

\section{Syntomic}
\label{sec:syntomic}

\subsection{Prismatization}

We begin by constructing some variants of the stacks of \cite{prismatization}.

\begin{recollection}
\label{rmk:deltaringNygmap}
For a $\delta$-ring $A$, there is a natural map of c-stacks
\[ \Spf A \times \ZpN \to A^\Nyg \]
defined by sending an $R$-point of $\Spf A \times \ZpN$, which consists of a filtered Cartier--Witt divisor $M \to W$ over $R$ together with a map of $\delta$-rings $A \to W(R)$, to the pair $(M \to W, A \to W(R) \to (W/M)(R))$.
\end{recollection}

\begin{remark}
\label{rmk:deltaringNygPrism}
Over the locus $\jHT : \ZpPrism \inj \ZpN$, the above construction yields the usual map
\[ \Spf A \times \ZpPrism \to A^\Prism \mpunc, \]
while over the locus $\jdR : \ZpPrism \inj \ZpN$ it yields its precomposite with the Frobenius on $A$: indeed, we have $W/\jdR(I) \isom F_*(W/I)$, so the map $A \to W/\jdR(I)$ factors as
\[ A \to W \xrightarrow{F} W \to W/I \mpunc, \]
which identifies with
\[ A \xrightarrow{F} A \to W \to W/I \mpunc. \]
\end{remark}

\begin{remark}
Note that the map $\Mm[\ZpN] \to \Mm^\Nyg$ of \Cref{rmk:deltaringNygmap} agrees by construction with the map $\Mm[\ZpN] \inj W_{\ZpN} \to \Mm^\Nyg$ induced by the Teichmüller embedding.
\end{remark}

\begin{construction}
\label{deltaringSynmap}
One can also give a syntomic variant of \Cref{rmk:deltaringNygmap} as follows.
We have two maps $\Spf A \times \ZpPrism \to \Spf A \times \ZpN$, one given by $\jdR$ and the other by the precomposite of $\jHT$ with the Frobenius on $A$.
By \Cref{rmk:deltaringNygPrism} these maps fit into commutative diagrams
\[\xymatrix{
  \Spf A \times \ZpPrism \ar[d] \ar[r] & \Spf A \times \ZpN \ar[d] \\
  A^\Prism \ar[r] & A^\Nyg
  \mpunc,
}\]
where the left vertical arrow is the precomposite of the usual map with the Frobenius on $A$.
We may therefore glue these two maps to obtain a c-stack $(A/A)^\Syn$ over $A^\Syn$.\footnotemark
\footnotetext{We expect that this stack is related to the relative syntomic cohomology defined in \cite[\S 7]{prismaticdelta}, but we do not pursue this here.}
\end{construction}

\begin{remark}
The c-stack $(A/A)^\Syn$ of \Cref{deltaringSynmap} is somewhat more tractable if $A$ is perfect, as the glued maps are then both open immersions.
\end{remark}

\begin{construction}
\label{prismNyg}
Let $(A, (d))$ be a prism, and write $A^{(1)}$ for $A$ considered as an $A$-algebra via the Frobenius.
Write $\phi^{-1}(d)$ for the natural distinguished element of $A^{(1)}$, so that $\phi : A \to A^{(1)}$ sends $d \mapsto \phi(\phi^{-1}(d))$.
We may associate to $A$ the stack
\[ (\overline{A} / A)^\Nyg \defeq \prn*{\Spf_{(p, d)} A^{(1)} \ang{u, t} / (u t - \phi^{-1}(d))} / \Gm \mpunc, \]
where $\overline{A} \defeq A/d$, and where $u$ is in weight $1$ and $t$ in weight $-1$.
Note that this stack comes with a canonical line bundle $L$ along with maps $u : \calO \to L$ and $t : L \to \calO$ whose composite is $\phi^{-1}(d)$.

We define as follows a natural filtered Cartier--Witt divisor $M \to W$ on $(\overline{A} / A)^\Nyg$.
The prism $(A^{(1)}, (\phi^{-1}(d)))$ gives us an invertible Cartier--Witt divisor $\phi^{-1}(d) W \to W$.
Pushing out along $u^\sharp : \phi^{-1}(d) \Gasharp \to L^\sharp$ yields
\[\xymatrix{
  0 \ar[r] & \phi^{-1}(d) \Gasharp \ar[d]^{u^\sharp} \ar[r] & \phi^{-1}(d) W \ar[d] \ar[r] & \phi^{-1}(d) F_* W \ar@{=}[d] \ar[r] & 0 \\
  0 \ar[r] & L^\sharp \ar@{-->}[d] \ar[r] & \pullbackcorner[ul][-1.8pc] M \ar@{-->}[d] \ar[r] & \phi^{-1}(d) F_* W \ar[d] \ar[r] & 0 \\
  0 \ar[r] & \Gasharp \ar[r] & W \ar[r] & F_* W \ar[r] & 0
  \mpunc.
}\]
Filling in the left dashed arrow with $t^\sharp$ makes the left vertical composite the map coming from $\phi^{-1}(d) W \to W$, so we obtain in the middle a factorization of $\phi^{-1}(d) W \to W$ through $M$.
It follows immediately that this is indeed a filtered Cartier--Witt divisor.

Note that as $W$ is an $A$-algebra scheme over $A$, the quotient $W/M$ is naturally an $\overline{A}$-algebra, so we have a natural map
\[ (\overline{A} / A)^\Nyg \to \fakeheight{\overline{A}}{A}^\Nyg \mpunc. \]
\end{construction}

\begin{remark}
\Cref{prismNyg} specializes in the perfect case to \cite[Example 5.5.6]{fgauges}.
\end{remark}

\begin{remark}
Note that the $u$-invertible locus of $(\overline{A} / A)^\Nyg$ is simply $\Spf A$, equipped with its usual Cartier--Witt divisor, while the $t$-invertible locus is also $\Spf A$, but with the Frobenius twist of its usual Cartier--Witt divisor.
We thus prefer to denote the latter by $\Spf A^{(1)}$, and we write
\begin{gather*}
  \jHT : \Spf A^{\phantom{(1)}} \inj (\overline{A} / A)^\Nyg \\
  \jdR : \Spf A^{(1)} \inj (\overline{A} / A)^\Nyg
\end{gather*}
for the corresponding inclusions.
\end{remark}

We have the following extension of \cite[Proposition 4.10]{prismaticdelta}.

\begin{lemma}
\label{relNygcompat}
For a prism $(A, (d))$, with $\overline{A} \defeq A/d$, the natural square
\[\xymatrix{
  (\overline{A} / A)^\Nyg \ar[d] \ar[r] & \fakeheight{\overline{A}}{A}^\Nyg \ar[d] \\
  \Spf A \times \ZpN \ar[r] & A^\Nyg
}\]
is a pullback.
\end{lemma}
\begin{proof}
A point of the pullback over an $A$-algebra $R$ consists of a filtered Cartier--Witt divisor $M \to W$ together with a factorization of the natural map $A \to W/M$ through $\overline{A}$.
Writing $M$ as an extension of $F_* I$ by $L^\sharp$, we find by \emph{loc.\ cit.\@} that $F_* I \to F_* W$ comes from an $R$-point of $\Spf_{(p, d)} A^{(1)}$, so we obtain from the factorization a diagram
\[\xymatrix{
  0 \ar[r] & \phi^{-1}(d) \Gasharp \ar[d]^{u^\sharp} \ar[r] & \phi^{-1}(d) W \ar[d] \ar[r] & \phi^{-1}(d) F_* W \ar@{=}[d] \ar[r] & 0 \\
  0 \ar[r] & L^\sharp \ar[d]^{t^\sharp} \ar[r] & \pullbackcorner[ul][-1.8pc] M \ar[d] \ar[r] & \phi^{-1}(d) F_* W \ar[d] \ar[r] & 0 \\
  0 \ar[r] & \Gasharp \ar[r] & W \ar[r] & F_* W \ar[r] & 0
}\]
as in \Cref{prismNyg}.
Now by \cite[Construction 5.2.2(3)]{fgauges} we uniquely lift
\[ \phi^{-1}(d) \Gasharp \xrightarrow{u^\sharp} L^\sharp \xrightarrow{t^\sharp} \Gasharp \]
to a sequence
\[ \phi^{-1}(d) \Ga \xrightarrow{u} L \xrightarrow{t} \Ga \mpunc, \]
yielding a lift of our $R$-point from $\Spf_{(p, d)} A^{(1)}$ to $(\overline{A} / A)^\Nyg$.
It is easy to see that this map is inverse to the natural map in the other direction.
\end{proof}

\begin{remark}
As a consequence of \Cref{relNygcompat}, given a $\delta$-pair $A \to R$, one can unambiguously define the relative filtered prismatization $(R/A)^\Nyg$ as the pullback
\[\xymatrix{
  (R/A)^\Nyg \pullbackcorner[dr][-1.6pc] \ar[d] \ar[r] & R^\Nyg \ar[d] \\
  \Spf A \times \ZpN \ar[r] & A^\Nyg
  \mpunc.
}\]
\end{remark}

\begin{remark}
It follows from \Cref{relNygcompat} that \Cref{prismNyg} is independent of the choice of $d$, and that it extends to non-orientable prisms.
One can also check this directly as in \cite[Example 5.5.6]{fgauges}.
\end{remark}

\begin{definition}[The complete filtered prismatization]
Recall that the filtered prismatization of a ring $R$ comes with the \emph{Rees map} $R^\Nyg \to \AomodGm$ encoding the Nygaard filtration.
We thus define the \emph{complete filtered prismatization} as the pullback
\[ R^\Nyghat \defeq R^\Nyg \times_{\AomodGm} \AohatmodGm \mpunc. \]
Note that the complement of $R^\Nyghat$ in $\ZpN$ is $\jdR(R^\Prism)$.

Similarly, given a $\delta$-pair $A \to R$, we define
\[ (R/A)^\Nyghat \defeq (R/A)^\Nyg \times_{\AomodGm} \AohatmodGm \mpunc. \]
\end{definition}

We will need the following result, which gives a very concrete description of $A^\Syn$ whenever $A$ is a perfect $\delta$-ring.
It is essentially equivalent (by \Cref{synrelaffine}) to \cite[Theorem 1.2(3)]{prismaticdelta},\footnotemark{} but the proof given below is instead adapted from \cite{shearedprismatization}, which proves the analogous result on the prismatic locus.
\footnotetext{This implies the g-stack version, from which the c-stack version follows, as $A^\Syn \to \ZpSyn$ is a relative g-stack.}
\begin{proposition}
\label{perfdeltaprismatization}
For a perfect $\delta$-ring $A$, the natural map of c-stacks
\[ (A/A)^\Syn \to A^\Syn \]
is an equivalence.
\end{proposition}
\begin{proof}
We may check this after pullback to $\ZpN$, where the statement is that
\[ \Spf A \times \ZpN \to A^\Nyg \]
is an equivalence.
We need to see that, given a filtered Cartier--Witt divisor $d : M \to W$ over a ring $R$, any map $A \to (W/M)(R)$ lifts uniquely to a $\delta$-ring map $A \to W(R)$.
By \Cref{perfdeltamapout} it will suffice to see that $W(R)^\flat \to (W/M)(R)^\flat$ is an equivalence, which follows from \Cref{WmodMflat}.
\end{proof}

\begin{proposition}
\label{WmodMflat}
If $M \to W$ is a filtered Cartier--Witt divisor, then the natural map
\[ W^\flat \to (W/M)^\flat \]
is an equivalence.
\end{proposition}
\begin{proof}
Recall that $M$ sits in an admissible sequence \cite[Remark 5.2.5]{fgauges}
\[ 0 \to K \to M \to F_* I \to 0 \]
where $I \to W$ is a Cartier--Witt divisor.
Note that the map $K \to F_* F^* K$ vanishes: as $K$ is locally isomorphic to $\ker F$ it suffices to show this for the latter, where it is clear.
The map $M \to F_* F^* M$ thus factors through $M \to F_* I$.
We therefore have maps
\[ W/M \to W/I \to W/F^*M \to W/F^*I \]
where the composites of the first two and of the second two both induce equivalences on $(-)^\flat$.\footnotemark
\footnotetext{This follows from \Cref{perfstatic} together with the fact that if $I \to R$ is a quasi-ideal in characteristic $p$, then we have a sequence $R/I \to R/\phi^* I \to R/I \to R/\phi^* I$ where the composites of the first two and second two are the Frobenius maps.}

We therefore reduce to the statement for $I$, which is contained in \cite{shearedprismatization}.
We sketch the argument here for convenience.
As above, it suffices to check after twisting by Frobenius.
Localizing, we may then assume by \Cref{twistCWdivtop} that $I = (p)$.
But $W \to W/p$ induces an isomorphism on $\pi_0(-/p)$, so we win by \Cref{perfstatic}.
\end{proof}

\begin{lemma}
\label{twistCWdivtop}
In an oriented prism $(A, (d))$, the element $\phi^n(d)$ is a unit multiple of $p$ modulo $d^{p^n}$, for any $n \in \bbN$.
\end{lemma}
\begin{proof}
The statement is clear for $n = 0$; assume by induction that $\phi^{n-1}(d) = p u + d^{p^{n-1}} r$ for some $u$.
Then we have
\begin{align*}
  \phi^n(d)
    &= \phi^{n-1}(p \delta(d) + d^p)
  \\&= p \phi^{n-1}(\delta(d)) + \phi^{n-1}(d)^p
  \\&= p \phi^{n-1}(\delta(d)) + d^{p^n} r^p + p^2 \cdot (\dots)
  \mpunc,
\end{align*}
so we conclude.
\end{proof}

\begin{lemma}
\label{perfdeltamapout}
For a perfect $\delta$-ring $A$ and $p$-complete animated ring $R$, the natural maps
\begin{align*}
  \Map_\delta(A, W(R))
    &\to \Map(A, W(R))
  \\&\to \Map(A, R)
  \\&\to \Map(A, R/p)
  \\&\to \Map_{\Fp}(A/p, R/p)
  \\&\from \Map_{\Fp}(A/p, R^\flat)
\end{align*}
are equivalences.
\end{lemma}
\begin{proof}
The second and third maps are equivalences by deformation theory (see e.g.\ \cite[Proposition 2.3]{bhattbms1notes}), the fourth trivially, and the last by perfectness of $A/p$.
As the composite of the first two maps is an equivalence, so is the first.
\end{proof}

\subsection{\texorpdfstring{$W$}{W}-module schemes}
\label{sec:Wmod}

Here we prove a number of general results on affine $W$-module schemes.
We first identify them with a full subcategory of graded group schemes,\footnotemark{} and we show that this embedding intertwines $W$-module and Cartier duality.
\footnotetext{See \cite{GaperfdR} for a similar study of $\Ga$- and $\Gaflat$-modules.}
We then prove a number of results on extensions of $W$-module schemes.

In this section, gradings are by default indexed by $\bbN$, and all $\Wbig$- and $W$- module schemes are assumed affine.

\begin{remark}
Note that the category of commutative graded affine group schemes is equivalent to the category of $\Mm$-equivariant commutative affine group schemes.
We will often regard $W$-module schemes as graded group schemes via the action of the Teichmüller embedding $\Mm \inj \Mm(W)$.
\end{remark}

\begin{definition}
We say that an $\Mm$-equivariant affine scheme is \emph{pointed} if the unit map to the weight $0$ piece of its coordinate ring is an isomorphism.
\end{definition}

\begin{lemma}
An $\Mm$-equivariant affine $R$-scheme $X$ is pointed if and only if $0 \in \Mm$ acts by $X \to * \to X$ for some $R$-point $* \to X$.
\end{lemma}
\begin{proof}
If $0$ acts in the specified way, then the $\Mm$-invariant functions are exactly the constants, which is exactly the definition of pointedness.
Conversely, note that the action of $0$ on $\calO(X)$ is given by the projection to weight $0$, as $t \in \Mm$ acts by $t^n$ in weight $n$.
Thus if $X$ is pointed then it factors through an $R$-point of $X$.
\end{proof}

The following result is a generalization of \cite[Proposition 2.1.14]{GaperfdR}.

\begin{lemma}
\label{Wbigmodstructureautomatic}
The forgetful functor from the category of $\Wbig$-module schemes to the category of $\Mm$-equivariant commutative affine group $R$-schemes is fully faithful, with $M$ in the essential image if and only if it is pointed.
\end{lemma}
\begin{proof}
First note that any $\Wbig$-module scheme is pointed in the above sense, as $0 \in \Mm$ acts by the zero map.

Conversely, given a pointed $\Mm$-equivariant affine group scheme, the action of $\Mm$ extends uniquely in $\Ind(\AffSch)$ to an action of the free commutative monoid on $(\Mm, 0)$, which is the rig $\Wratplus$ of positive rational Witt vectors (see \Cref{rmk:Wratplus} or \cite[Construction A.4]{akhilrh}).
Group completing, we may regard this as a $\Wrat$-module structure in $\bbZ[\Mm]$-modules in abelian presheaves.
This structure is encoded by a map $\Wrat \tensor_{\bbZ[\Mm]} M \to M$ such that the two induced maps $\Wrat \tensor_{\bbZ[\Mm]} \Wrat \tensor_{\bbZ[\Mm]} M \rightrightarrows M$ coincide.
Applying \Cref{SMmWratW} and the tensor-$\Hom$ adjunction, we find that this is the same as the data of a $\Wbig$-module structure.
By naturality of the above constructions, we obtain a functor from pointed $\Mm$-equivariant affine group schemes to $\Wbig$-module schemes, which is by construction inverse to the forgetful functor.
\end{proof}

In the proof above we used only the affine scheme case of the following lemma, but we will need the general statement later.

\begin{lemma}
\label{SMmWratW}
If $M$ is a commutative group in $\Mm$-equivariant affine stacks (see \Cref{sec:affstk}), then the natural map
\[ \intHom_{\bbS[\Mm]}(\Wbig, M) \to \intHom_{\bbS[\Mm]}(\Wrat, M) \]
is an equivalence.
\end{lemma}
\begin{proof}
The groups in question compute $\Mm$-equivariant commutative monoid homomorphisms from $\Wbig$, resp.\ $\Wratplus$.
By \Cref{mapobjcalg} it will suffice to see for each $n$ that $(\Wratplus)^n \to (\Wbig)^n$ induces an equivalence after taking $\Mm$-equivariant maps of stacks to $M$.
\Cref{gradedDAlgwelldefined} and affineness of $M$ then reduce us to showing that the induced maps on graded global sections are equivalences.
But the natural map of graded rings $\calO((\Wbig)^n) \to \calO((\Wratplus)^n)$ is an isomorphism, as can be seen by the description of both sides in terms of symmetric functions.
\end{proof}

\begin{remark}
As the $n$-Frobenius on $\Wbig$ commutes (under the Teichmüller embedding) with with the $n$-power map on $\Mm$, pushforward of a $\Wbig$-module along the $n$-Frobenius corresponds on the coordinate ring to scaling the grading by $n$.
\end{remark}

\begin{remark}
\label{Wmodstructureautomatic}
Over a $p$-local base we have a natural sequence of ring schemes $W \to \Wbig \to W$ which splits off $W$ as a factor of $\Wbig$ \cite[Proposition 10]{larswitt}.
In particular, restriction of scalars along the second map is fully faithful.
As this map is compatible with Teichmüller lifts, we see from \Cref{Wbigmodstructureautomatic} that the forgetful functor from affine $W$-module schemes to pointed $\Mm$-equivariant affine group schemes is fully faithful.
Note, on the other hand, that the first map $W \to \Wbig$ is \emph{not} $\Mm$-equivariant.
\end{remark}

\begin{definition}
Duality gives an antiautoequivalence of the category of weightwise finite free graded Hopf algebras over a given base, which we refer to as \emph{(graded) Cartier duality}.
Note that we take the weightwise dual here, so our objects remain $\bbN$-graded.
\end{definition}

\begin{theorem}[Cartier \cite{cartierW}, {\cite[Lemma 2.4]{wickelgrenW}}]
\label{Wbigselfdual}
The graded affine group scheme $\Wbig$ is self-Cartier-dual.
\end{theorem}

\begin{lemma}
\label{WbigmodCartierduality}
Let $M$ be a $\Wbig$-module scheme.
Then there is a natural isomorphism of $\Wbig$-modules
\[ M^\dual \bij \intHom_{\Wbig}(M, \Wbig) \mpunc, \]
where the $\Wbig$-module structure on the source is obtained from \Cref{Wbigmodstructureautomatic}.
\end{lemma}
\begin{proof}
An $S$-point of $M^\dual$ is the same as a $\Wbig$-module homomorphism $\Wbig \to M^\dual$ over $S$; dualizing and applying \Cref{Wbigselfdual} and \Cref{Wbigmodstructureautomatic} we obtain a homomorphism $M \to \Wbig$.
The same steps in reverse give the inverse.
\end{proof}

\begin{corollary}
\label{WmodCartierduality}
Let $M$ be a $W$-module scheme over a $p$-local base.
Then there is a natural isomorphism of $\Wbig$-modules
\[ M^\dual \bij \intHom_W(M, W) \mpunc. \]
In particular, the source is obtained from restriction of scalars along $\Wbig \to W$.
\end{corollary}
\begin{proof}
This follows from \Cref{WbigmodCartierduality} along with the observation that, as $\Wbig \to W$ is the projection onto a factor of a product, restriction of scalars commutes with duality.
\end{proof}

We will thus make no distinction between $W$-module and Cartier duality in the sequel.

\begin{recollection}
\label{VFdual}
By the results of \cite[\S 3.8]{prismatization}, $W$-module duality over a $p$-local base swaps the two exact sequences
\begin{gather*}
  0 \to F_* W \xrightarrow{V} W \to \Ga \to 0 \\
  0 \to \Gasharp \to W \xrightarrow{F} F_* W \to 0
  \mpunc.
\end{gather*}
In particular, we may identify $F_* W$ with its dual.
\end{recollection}

\begin{corollary}
\label{WmodExt1}
Over a $p$-local base, we have
\[ \intExt^1_W(\Ga, W) = \intExt^1_W(F_* W, W) = 0 \mpunc. \]
\end{corollary}
\begin{proof}
Applying $\intHom_W(-, W)$ to the short exact sequences of \Cref{VFdual}, we find via the long exact sequences that the two groups in question inject into $\intExt^1_W(W, W)$, which vanishes.
\end{proof}

\begin{lemma}
\label{wflfextn}
Let
\[ 0 \to K \to M \to N \to 0 \]
be an extension of $\Wbig$-module schemes over a ring $R$.
Then if $K$ and $N$ are weightwise finite locally free, then so is $M$.
In this case, the rank of $\calO(M)$ in a given weight is the same as that of $\calO(K \times N)$.
\end{lemma}
\begin{proof}
Write $\calO(N)_{\leq n}$ for the truncation of $\calO(N)$ to weights at most $n$.
Note that the natural map $\calO(M) \to \calO(M) \tensor_{\calO(N)} \calO(N)_{\leq n}$ is given by killing a collection of elements in weights greater than $n$, and is thus an isomorphism in weights up to $n$.
It thus suffices to see for each $n$ that $\calO(M) \tensor_{\calO(N)} \calO(N)_{\leq n}$ is weightwise finite locally free of the appropriate ranks.
The grading on the $\calO(N)$-module $\calO(N)_{\leq n}$ equips it with a finite filtration whose graded pieces are weight-shifts of restrictions of finite locally free $R$-modules along the augmentation $\calO(N) \to R$.
Weightwise finite local freeness of $\calO(M) \tensor_{\calO(N)} \calO(N)_{\leq n}$ thus reduces to the same for $\calO(M) \tensor_{\calO(N)} R \isom \calO(K)$, which holds by assumption.

To see the last statement, note that the same argument equips $\calO(K \times N) \tensor_{\calO(N)} \calO(N)_{\leq n}$ with a finite filtration with the same graded pieces as those of $\calO(M) \tensor_{\calO(N)} \calO(N)_{\leq n}$.
\end{proof}

\begin{lemma}
\label{Gasharpdualexact}
Suppose we are given an extension
\[ 0 \to \Gasharp \to M \to N \to 0 \]
of weightwise finite locally free $W$-modules over a $p$-local base $R$.
Then the dual sequence
\[ 0 \to N^\dual \to M^\dual \to \Ga \]
is right exact.
\end{lemma}
\begin{proof}
We need to show that $\calO(\Ga) \to \calO(M^\dual)$ is faithfully flat.
Write $\calO(\Ga) = R[t]$.
Note that $\calO(M^\dual)/t \isom \calO(N^\dual)$ is faithfully flat over $R$.
Faithful flatness of $\calO(N)[\tfrac{1}{t}]$ over $R[t^{\pm 1}]$ follows from faithful flatness over $R$, since these are $\bbZ$-graded rings with invertible elements in weight $1$, hence tensored up from weight $0$.
It thus suffices by \Cref{connflatstrat} to show that multiplication by $t$ is injective on $\calO(M^\dual)$.
We have for each $n$ an exact sequence
\[ \calO(M^\dual)_{n-1} \xrightarrow{t} \calO(M^\dual)_n \to \calO(N^\dual)_n \to 0 \mpunc. \]
But by \Cref{wflfextn} the ranks of the left and right terms sum to that of the middle term, so the sequence is left exact.
\end{proof}

\begin{lemma}
\label{connflatstrat}
Let $R$ be a ring with an element $t \in \pi_0 R$.
Then $M \in \calD(R)^{\leq 0}$ is flat, resp.\ faithfully flat, if and only if it is after base change to $R/t \times R[\tfrac{1}{t}]$.
\end{lemma}
\begin{proof}
The ``only if'' direction is clear.
Flatness in the ``if'' direction is \citestacks{0H85}, while faithfulness follows from flatness along with \cite[Lemma 5.2.2]{purityflatcoh} and the fact that base change to $R {/^\rmL} t \times R[\tfrac{1}{t}]$ is conservative.
\end{proof}

\begin{proposition}
\label{GasharpExt1}
Over a $p$-local base, we have
\[ \intExt^1_W(\Gasharp, W) = 0 \mpunc. \]
\end{proposition}
\begin{proof}
We have an exact sequence
\[ \intHom_W(\Gasharp, F_* W) \to \intExt^1_W(\Gasharp, \Gasharp) \to  \intExt^1_W(\Gasharp, W) \to  \intExt^1_W(\Gasharp, F_* W) \mpunc, \]
where the first term vanishes by \cite[Proposition 5.2.1(3)]{fgauges}.
We claim that the last term vanishes as well.

To see this, suppose we have an extension
\[ 0 \to F_* W \to M \to \Gasharp \to 0 \mpunc. \]
Dualizing, we obtain a sequence
\[ 0 \to \Ga \to M^\dual \to F_* W \mpunc. \]
By \Cref{wflfextn}, $\calO(M^\dual)$ is of rank $1$ in weight $1$, so the map of graded rings $\calO(M^\dual) \to \calO(\Ga)$ admits a unique splitting, which is necessarily a Hopf map by grading considerations.
Dualizing back, we obtain a splitting of our original sequence.

We have thus reduced to showing that $\intExt^1_W(\Gasharp, \Gasharp)$ vanishes.
Suppose as before that we have an extension
\[ 0 \to \Gasharp \to M \to \Gasharp \to 0 \mpunc. \]
Dualizing, we obtain an exact sequence
\[ 0 \to \Ga \to M^\dual \to \Ga \to 0 \]
by \Cref{Gasharpdualexact}.
Note that the induced sequence of coordinate rings is exact in weight $1$, so the weight $1$ piece of $\calO(M^\dual)$ admits generators $x$ and $y$ such that
\begin{gather*}
  \Delta(x) = x \tensor 1 + 1 \tensor x \\
  \Delta(y) = y \tensor 1 + 1 \tensor y + a \cdot (x \tensor 1 + 1 \tensor x)
  \mpunc.
\end{gather*}
The map $\calO(\Ga) \to \calO(M^\dual)$ sending the generator to $y - a x$ therefore yields the desired splitting.
\end{proof}

\begin{remark}
We do not know whether the group $\intExt^1_W(M, W)$ vanishes whenever $M$ is weightwise finite locally free.
This would follow from the analogue of \Cref{Gasharpdualexact} for $W$ in place of $\Gasharp$.
\end{remark}

\begin{lemma}
\label{admissibleWmodinvopen}
Let $M$ be an extension of $F_* W$ by $\Gasharp$ in $W$-modules over a $p$-nilpotent ring $R$.
Then $M$ is isomorphic to $W$ if and only if its extension class lies in
\[ \Gm^\dR \inj \Ga^\dR \isom \intExt_W(F_* W, \Gasharp) \mpunc, \]
where the identification on the right is \cite[Proposition 5.2.1(2)]{fgauges}.
\end{lemma}
\begin{proof}
The argument of \emph{loc.\ cit.\@} shows that such an extension is given locally by pushing out the usual one (coming from the Frobenius) along an endomorphism of $\Gasharp$.
Such an endomorphism is given by an element of $\Ga$, and if the endomorphism is invertible then the pushout clearly remains isomorphic to $W$.
As $\Gm \inj \Ga$ consists of invertible endomorphisms, and the preimage of $\Gm^\dR$ in $\Ga$ is $\Gm$, we obtain the ``if'' direction.
For the ``only if'' direction, note that if the extension class does not lie in this locus, then after pullback to some perfect-field-valued point of $R$ the extension splits, in which case $M$ is not isomorphic to $W$.
\end{proof}

\begin{remark}
Note that $\calO(W)$ is a polynomial ring on a single generator of each $p$-power weight.
Indeed, this follows from the description of $W$ in terms of Witt coordinates, together with the identity $[a] \cdot V^n [b] = V^n [a^{p^n} b]$.
The Teichmüller map $\bbA^1 \to W$ then exhibits $\calO(\bbA^1)$ as the quotient of $\calO(W)$ by all polynomial generators of weight greater than $1$.
\end{remark}

The following \namecref{gradedpolyquotient} will be very useful for relating quasi-ideals over $W$ to their pullbacks along the Teichmüller map $\Mm \to W$.

\begin{lemma}
\label{gradedpolyquotient}
Let $R$ be a ring, $S$ a graded polynomial $R$-algebra with finitely many generators in each weight and none in weight $0$, and $M$ an object of the graded derived category of $S$.
Then $M$ is connective, resp.\ perfect over $R$, in each weight if and only if $M \tensor_S R$ is.
Furthermore, if $M \tensor_S R$ is of $\Tor$-amplitude $[m, n]$ over $R$ in each weight, then so is $M$.
\end{lemma}
\begin{proof}
The tensor product is computed in each weight by a finite colimit of the graded pieces of $M$, implying the forward direction.
For the converse, write $S_n$ for the quotient of $S$ by the elements of weight greater than $n$.
Then $M \to M \tensor_S S_n$ is an equivalence in weights up to $n$.
But the $S$-module $S_n$ has a finite filtration whose graded pieces are finite sums of weight-shifts of $R$, yielding the result.
\end{proof}

\subsection{Passable \texorpdfstring{$W$}{W}-modules and polyfiltered Cartier--Witt divisors}
\label{sec:passable}

In this section we introduce variants of the notions of admissible $W$-module and filtered Cartier--Witt divisor which are necessary for the proof of \Cref{ZpSynmainthm} (see \Cref{rmk:polyfilteredexplanation}).

\begin{definition}
We say that an affine $W$-module scheme $M$ over a $p$-local base is \emph{$0$-passable} if it is invertible.
For $n > 0$, we say that it is \emph{$n$-passable} if it fits in an exact sequence
\[ 0 \to L^\sharp \to M \to F_* M' \to 0 \]
with $L$ a line bundle and $M'$ an $(n-1)$-passable $W$-module.
In this situation, we refer to a sequence of the form above as a \emph{passable sequence}.
We say that $M$ is \emph{passable} if it is $n$-passable for some $n$.
\end{definition}

\begin{lemma}
The passable sequence associated to a passable $W$-module is unique and functorial.
Passability is fpqc-local.
\end{lemma}
\begin{proof}
This follows as in \cite[Remark 5.2.5]{fgauges} from the claim that the module $\intHom_W(L^\sharp, F_* N)$ vanishes for any weightwise finite locally free $W$-module $N$.
For the claim, it suffices by duality to see that $\intHom_W((F_* N)^\dual, L^\dual)$ vanishes.
But this is clear, as the source is trivial in weight $1$, while the target is generated in weight $1$.
\end{proof}

\begin{definition}
\label{vfCWdiv}
We say that a map of $W$-modules $M \to W$ over a $p$-nilpotent base is a \emph{$n$-polyfiltered Cartier--Witt divisor} if there exists some $n \geq 0$ such that:
\begin{enumerate}
\item
  $M$ is $n$-passable; and
\item
  \label{vfCWdiv:mapcond}
  the base is covered by finitely many open subsets along each of which the restriction of $M \to W$ either:
  \begin{enumerate}[ref=\alph*]
  \item
    \label{vfCWdiv:mapcond:filt}
    is a filtered Cartier--Witt divisor; or
  \item
    \label{vfCWdiv:mapcond:twist}
    fits into a pullback square
    \[\xymatrix{
      M \pullbackcorner \ar[d] \ar[r] & F_* M' \ar[d] \\
      W \ar[r]^F & F_* W
    }\]
    with $M' \to W$ an $(n-1)$-polyfiltered Cartier--Witt divisor.
  \end{enumerate}
\end{enumerate}
We say that $M \to W$ is a \emph{polyfiltered Cartier--Witt divisor} if it is $n$-polyfiltered for some $n$, and we let $\ZpNN$ denote the moduli stack of polyfiltered Cartier--Witt divisors.
It follows from \cite[Proposition 5.3.8]{fgauges} that a polyfiltered Cartier--Witt divisor is always a quasi-ideal, so we obtain for any ring $R$ a stack $R^\NNyg$ via the evident transmutation.
\end{definition}

\begin{remark}
Note that a $1$-passable $W$-module is simply an admissible $W$-module, and that a $1$-polyfiltered Cartier--Witt divisor is simply a filtered Cartier--Witt divisor.
\end{remark}

\begin{remark}
\label{rmk:polyfilteredexplanation}
The relevance of \Cref{vfCWdiv} is that there are many quasi-ideals in $W$ which do not come from $\ZpN$ but whose associated ring stacks do come from $\ZpSyn$.\footnotemark
\footnotetext{The methods of \Cref{ZpSynmainthm} can be used to show that $\ZpNN$ is indeed the full preimage of $\ZpSyn$ in the stack of $W$-algebra stacks.}
Similarly, and more directly relevant to our situation, $n$-polyfiltered Cartier--Witt divisors admit specializations to $(n+1)$-polyfiltered such, and some of these specializations have constant associated monoid stack.
We therefore have no choice but to consider $2$-polyfiltered Cartier--Witt divisors in the proof of \Cref{ZpSynmainthm}.
\end{remark}

\begin{remark}
One can check that $\ZpNN$ is an $\bbN$-indexed chain of copies of $\ZpN$ glued pairwise along $\ZpN \xleftarrow{\jHT} \ZpPrism \xrightarrow{\jdR} \ZpN$.
\end{remark}

\begin{remark}
\label{rmk:vfCWTR}
A variant of $(-)^\NNyg$ in positive characteristic is independently studied in work of Yuanning Zhang, who relates its cohomology to de Rham--Witt forms and thus topological restriction homology \cite{yuanningdRW}.
We expect that the cohomology of $(-)^\NNyg$ is similarly related to $\TR^{\rmh S^1}$.
\end{remark}

\begin{lemma}
If $M \to W$ is a polyfiltered Cartier--Witt divisor, then $W/M$ arises from a point of $\ZpSyn$.
\end{lemma}
\begin{proof}
It suffices to show this locally on the base.
By definition we have a finite cover of the base by open sets along each of which $M \to W$ is a Frobenius twist of a filtered Cartier--Witt divisor, so we conclude.
\end{proof}

\begin{lemma}
\label{fCWpointwise}
A quasi-ideal $M \to W$ over a $p$-nilpotent ring $R$ with $M$ admissible is a filtered Cartier--Witt divisor if and only if it is a polyfiltered Cartier--Witt divisor at every field-valued point of $R$.
\end{lemma}
\begin{proof}
This follows from \cite[Remark 5.2.5]{fgauges} and \cite[Proposition 5.1.2]{fgauges}.
\end{proof}

The following results are key to the proof of \Cref{ZpSynmainthm}, as they allow us to check stratum-by-stratum that a ring structure on a given monoid stack is of the desired form.

\begin{lemma}
\label{vfCWpointwise}
A map of $W$-modules $M \to W$ over a $p$-nilpotent ring $R$ with $M$ passable is a polyfiltered Cartier--Witt divisor if and only if it is at every field-valued point of $R$.
\end{lemma}
\begin{proof}
The ``only if'' direction is obvious.
For the ``if'' direction, we suppose $M$ is $n$-polyfiltered and proceed by induction on $n$.
If $n = 1$ then the statement follows from \cite[Remark 5.2.5]{fgauges} and \cite[Proposition 5.1.2]{fgauges}.

Now suppose we have shown the statement for $(n-1)$-polyfiltered Cartier--Witt divisors.
The map $M \to W$ yields a map of passable sequences
\[\xymatrix{
0 \ar[r] & L^\sharp \ar[d] \ar[r] & M \ar[d] \ar[r] & F_* M' \ar[d] \ar[r] & 0 \\
0 \ar[r] & \Gasharp \ar[r] & W \ar[r] & F_* W \ar[r] & 0
\mpunc,
}\]
where $M' \to W$ is an $(n-1)$-polyfiltered Cartier--Witt divisor by the inductive hypothesis.
We need to produce an open cover of $R$ satisfying the conditions of \Cref{vfCWdiv}(\ref{vfCWdiv:mapcond}).

Satisfying (\ref{vfCWdiv:mapcond:filt}) amounts to the condition that $M'$ is actually invertible rather than simply passable.
This is an open condition by \Cref{admissibleWmodinvopen} (applied $n-1$ times, from the inside out).

Satisfying (\ref{vfCWdiv:mapcond:twist}) amounts to the condition that the left vertical map above is an isomorphism.
As we have $\intEnd_W(\Gasharp) \isom \Ga$ by \cite[Proposition 5.2.1(1)]{fgauges}, this is an open condition as well.

Each point of $\Spec R$ by assumption lies in one of these two open sets, so we conclude.
\end{proof}

\begin{proposition}
\label{passableWmodstrat}
Let $M$ be an affine $W$-module scheme over a $p$-local ring $R$, and take any element $t \in R$.
Suppose that $M$ is weightwise finite locally free and that $M|_{R/t}$ and $M|_{R[\tfrac{1}{t}]}$ are both $n$-passable.
Then so is $M$.
\end{proposition}
\begin{proof}
First suppose $n = 0$, so the restrictions of $M$ are invertible.
Localizing, we may replace $R$ by $R_{(t)} \times R[\tfrac{1}{t}]$.
On the second factor $M$ is invertible by assumption, so it remains to prove the statement on the first.
Localizing again, we may assume by \cite[Proposition 3.2.3]{apc} that $M|_{R/t}$ is isomorphic to $W$.
Lifting the polynomial generators, we obtain a map of graded rings $\calO(W) \to \calO(M)$, which is an isomorphism by Nakayama's lemma.
Thus $M$ is an infinite-dimensional affine space, so we may pick a lift of the point $1 \in M|_{R/t}$ to $M$; the $W$-module structure then yields a homomorphism $W \to M$.
By construction this is an isomorphism on $R/t$, so Nakayama implies that it is an isomorphism on $R$ itself.

Now assume by induction that we have shown the statement for $(n-1)$-passable modules.
Note that the weight $1$ piece of $\calO(M)$ is a line bundle $L$.
We thus obtain a ring homomorphism $\Sym^* L^\dual \to \calO(M^\dual)$, which in fact must be a Hopf algebra homomorphism.
After restriction to $R/t$ and $R[\tfrac{1}{t}]$ this map is Cartier dual to the inclusion $L^\sharp \inj M$ coming from the passable sequence, so by \Cref{Gasharpdualexact} and \Cref{connflatstrat} it is faithfully flat.
Thus $\calO(M^\dual) \tensor^\rmL_{\Sym^* L^\dual} R$ is flat, hence weightwise finite locally free by \Cref{gradedpolyquotient}, and it is concentrated in weights divisible by $p$, as can be checked modulo $t$.
Now the inductive hypothesis (after dividing weights by $p$) tells us that its Cartier dual is the Frobenius twist of a passable $W$-module $M'$.
Dualizing back and applying \Cref{WmodExt1}, we obtain the desired passable sequence
\[ 0 \to L^\sharp \to M \to F_* M' \to 0 \mpunc. \qedhere \]
\end{proof}

\begin{corollary}
\label{vfCWdivstrat}
Let $M \to W$ be a quasi-ideal scheme over a $p$-nilpotent ring $R$, and take any element $t \in R$.
Suppose that $M$ is weightwise finite locally free and that $M|_{R/t}$ and $M|_{R[\tfrac{1}{t}]}$ are both polyfiltered Cartier--Witt divisors.
Then so is $M$.
\end{corollary}
\begin{proof}
This follows immediately from \Cref{passableWmodstrat} and \Cref{vfCWpointwise}.
\end{proof}

\subsection{Covering \texorpdfstring{$\Mm^\Nyg$}{M\_m{\textasciicircum}N}}
\label{sec:coveringMmN}

We collect here some results on the map $\Mm[\ZpN] \to \Mm^\Nyg$ which are needed for \Cref{ZpSynmainthm}.

\begin{proposition}
\label{MmMmNflatness}
The map $\Mm[\ZpN] \to \Mm^\Nyg$ is faithfully flat.
\end{proposition}
\begin{proof}
We may check this after pullback to $\QdRNyg \times_{\BGm} \Spf \Zp$.
As $\Mm^\Nyg$ can be written as a quotient $W/M$ of flat affine group schemes, we see by \Cref{stackflatstrat} that we can check the statement modulo $(p, q^{1/p} - 1)$, i.e.\ after base change to $\Fp^\Nyg \tensor \Fp$.
Again by \Cref{stackflatstrat} we reduce to checking separately on $\jHT(\Fp^\dR)$, $\jdR(\Fp^\dR)$, and $\Fp^\Hodge$, where the statement is clear from the definitions.
\end{proof}

\begin{corollary}
\label{MmMmNkerfiniteness}
The $\Mm$-equivariant $\ZpN$-scheme $\Mm[\ZpN] \times_{\Mm^\Nyg} \ZpN$ is weightwise finite locally free.
\end{corollary}
\begin{proof}
First note that $W_{\ZpN} \times_{\Mm^\Nyg} \ZpN$ is an admissible $W$-module, hence weightwise finite locally free by \Cref{wflfextn}.
Thus by \Cref{gradedpolyquotient} we find that the derived pullback $\Mm[\ZpN] \times^\rmL_{\Mm^\Nyg} \ZpN$ is weightwise perfect.
But \Cref{MmMmNflatness} implies that this pullback is flat, so it must be weightwise finite locally free.
\end{proof}

\begin{remark}
\label{rmk:MmMmNflatdescendable}
Note that the map $\Mmperf[\ZpN] \to \Mm^\Nyg$ is countably-presented: as $\Mmperf \to W$ is, it suffices to see that $W_{\ZpN} \to \Mm^\Nyg$ is, which follows from the fact that the coordinate ring of an admissible $W$-module is.
Thus by \Cref{MmMmNflatness} and \cite[Corollary 3.33]{descendability} it is flat descendable.
\end{remark}

We used the following \namecref{stackflatstrat} above.

\begin{lemma}
\label{stackflatstrat}
Let $R$ be a ring with an element $t \in R$, and let $X$ be an $R$-stack with affine diagonal which admits a faithfully flat cover by a flat affine $R$-scheme $Y$.
Then a map $Z \to X$ from a flat affine $R$-scheme is flat, resp.\ faithfully flat, if and only if it is after base change to $R/t \times R[\tfrac{1}{t}]$.
\end{lemma}
\begin{proof}
The ``only if'' direction is clear, so we prove the ``if'' direction.
Flatness of $Y \to X$ implies that the pullback $Y \times_X Z$ is flat over $Z$, hence over $R$.
Since $Y \to Z$ is faithfully flat, flatness, resp.\ faithful flatness, of $Z \to X$ is equivalent to that of $Y \times_X Z \to Y$.
By \Cref{connflatstrat} and $R$-flatness of $Y \times_X Z$ and $Y$ this can be checked after base change to $R/t \times R[\tfrac{1}{t}]$, where it holds by assumption.
\end{proof}

\subsection{From monoids to rings}
\label{sec:monoidstorings}

The results of this section will be very useful for bootstrapping up from monoid structures to ring structures.
The first three\todo*[double-check in case order changed] results below are well-known, but we record them here for convenience.
\Cref{Mmringunique,Mmperfringunique} show that the ring structures on $\Mm$ and $\Mmperf$ are unique, which is the first step of the bootstrap, while \Cref{WbigperftoMmperf} lets us pass from $\Mmperf$ to $\Wbigperf$.

\begin{lemma}
\label{ppowerhigherhomotopy}
For an animated ring $R$, the $p$-power map on $\Mm(R)$ induces a map divisible by $p$ on higher homotopy groups at each point.
\end{lemma}
\begin{proof}
The $p$-power map on higher homotopy at $t \in R$ factors as
\[ \pi_i(R, t) \to \pi_i(R, t)^{\oplus p} \to \pi_i(R, t^p) \mpunc, \]
where the first map is the diagonal and the second is induced by the product.
But the second map is symmetric, so its value on $(x, \dots, x)$ agrees with that on $(p x, 0, \dots, 0)$.
\end{proof}

\begin{corollary}
\label{perfstatic}
If $R$ is an animated $\Fp$-algebra, then the map $R^\perf \to \pi_0(R)^\perf$ is an equivalence.
\end{corollary}
\begin{proof}
By \Cref{ppowerhigherhomotopy} the Frobenius on $R$ induces the zero map on higher homotopy groups, so in the limit they vanish.
\end{proof}

\begin{lemma}
\label{Mmperfvsflat}
For a $p$-complete animated ring $R$, the natural map
\[ \Mm(R)^\perf \to \Mm(R/p)^\perf = \Mm(R^\flat) \]
is an equivalence.
In particular, the source is static.
\end{lemma}

Note that if $R$ is static then this is essentially \cite[Lemma 3.4(i)]{perfectoid}.

\begin{proof}
It suffices to see that for any perfect monoid $M$ the natural map
\[ \Map(\bbZ[M], R) \to \Map(\bbZ[M], R/p) \]
is an equivalence.
But this follows from \Cref{perfdeltamapout}.
\end{proof}

\begin{lemma}
\label{Mmringunique}
Over any base ring, the scheme $\bbA^1$ admits a unique group structure equivariant under the action of $\Gm$.
The monoid $\Mm$ therefore admits a unique ring structure.
\end{lemma}
\begin{proof}
The coproduct on the graded coordinate ring of $\bbA^1$ is determined by what it does in weight $1$.
As the weight $0$ and $1$ pieces are free of rank $1$, the map is uniquely determined by the axioms of a Hopf algebra.
\end{proof}

\begin{lemma}
\label{Mmperfringunique}
Over a $p$-complete ring $R$, the formal monoid scheme $\Mmperf$ admits a unique ring structure.
\end{lemma}
\begin{proof}
Note that by \Cref{perfdeltamapout} the map $\Mmperf(S) \to \Mmperf(S/p)$ is an isomorphism for any $p$-complete $R$-algebra $S$, so we may replace $R$ by $R/p$.
We claim that there exists some $R \to S$ faithfully flat such that $S^\perf$ is absolutely integrally closed.
Indeed, by André's lemma \cite[Theorem 7.14]{prisms} there exists a faithfully flat map $R^\perf \to \tilde{S}_0$ with $\tilde{S}_0$ absolutely integrally closed; letting $S_0 \defeq \tilde{S}_0 \tensor_{R^\perf} R$, we find that $R^\perf \to S_0^\perf$ factors through $\tilde{S}_0$, so all monic polynomials in $R^\perf$ split in $S_0^\perf$.
Transfinitely iterating and taking the colimit along $\omega_1$, we obtain the desired faithfully flat $R$-algebra $S$.\footnotemark
\footnotetext{Note that perfection, as a countable limit, commutes with $\omega_1$-filtered colimits.}

Now suppose we have a ring structure $\calR$ on $\Mmperf[R]$.
Then, as $\Mmperf(S) = \Mm(S^\perf)$, we find by \Cref{charpMm} that $\calR(S)$ is an $\Fp$-algebra.
By faithful flatness, $\calR(R)$ injects into $\calR(S)$, so we find that $\calR$ is an $\Fp$-algebra scheme.

The $p$-power map and its inverse on $\calR$ are therefore ring endomorphisms, and in particular group endomorphisms.
Note that the Frobenius is a (Frobenius-linear) group endomorphism as well.
Thus the composite of the Frobenius with the inverse $p$-power map, which is the endomorphism of $\calR$ given by the Frobenius on $R$ and the identity on $\Mmperf$, is also a group endomorphism.
It therefore commutes with the addition map on $\calR$, which implies that the group structure is defined over the Frobenius-fixed points $R^{\phi = 1}$.
As this is a perfect ring, we conclude by \Cref{Mmperfringuniqueperf}.
\end{proof}

\begin{lemma}
\label{charpMm}
Let $R$ be a radically-closed\footnotemark{} perfect $\Fpbar$-algebra, and let $S$ be a ring equipped with an isomorphism $\Mm(R) \isom \Mm(S)$.
\footnotetext{That is, any polynomial $x^n - a$ in $R[x]$ splits into linear factors.}
Then $S$ is of characteristic $p$.
\end{lemma}
\begin{proof}
Note that $S$ is reduced, as $R$ is, so by \citestacks{00EW} it suffices to show that $S_\frakp$ is of characteristic $p$ for each minimal prime ideal $\frakp$ of $S$.
Note also that the notion of a multiplicative subset of a ring and the corresponding localization depend only on the underlying multiplicative monoid, so we have an $R$-algebra $(R \setminus \frakp)^{-1} R$ whose multiplicative monoid is isomorphic to that of $S_\frakp$.
In particular, we find that $\Gm(S_\frakp)$ contains a copy of $\Fpbar^\times \isom \bbQ / \bbZ[\tfrac{1}{p}]$.
By \citestacks{00EU} we know that $S_\frakp$ is a field, so we immediately see that its characteristic is either $0$ or $p$.
Suppose it is of characteristic $0$.
Then it is closed under prime-to-$p$ radical extensions, as its multiplicative group is divisible and contains all prime-to-$p$ roots of unity, hence under prime-to-$p$ solvable extensions.
It must therefore contain a primitive $p$\textsuperscript{th} root of unity.
But this contradicts perfectness of $R$.
\end{proof}

\begin{lemma}
\label{Mmperfringuniqueperf}
Over a perfect ring $R$ of characteristic $p$, the scheme $\Aoperf$ admits a unique group structure equivariant under the action of $\Gmperf$.
\end{lemma}
\begin{proof}
Because $R$ is perfect, a coproduct on the $\bbZ[\tfrac{1}{p}]$-graded ring $\calO(\Aoperf_R) \isom R[t^{1/p^\infty}]$ is determined by a group law
\[ F(t_1, t_2) = \sum_{i \in \bbN[\tfrac{1}{p}] \cap [0, 1]} a_i t_1^i t_2^{1-i} \]
where $a_0 = a_1 = 1$ and all but finitely many $a_i$ vanish.
Let $i$ be the smallest nonzero value such that $a_i$ is nonzero, and write $i = \tfrac{n}{p^k}$ with $p \ndivides n$.
Then in $F(F(t_1, t_2), t_3)$ the coefficient of $t_3^{1-i}$ is
\[ a_i F(t_1, t_2)^i = a_i t_2^i + a_i \cdot n a_i^{1 / p^k} t_1^{i / p^k} t_2^{(n - i) / p^k} + \cdots \mpunc, \]
so the coefficient of $t_1^{i / p^k}$ is nonzero.
On the other hand, in $F(t_1, F(t_2, t_3))$ no power of $t_1$ smaller than $t_1^i$ appears.
Thus by associativity we must have $i = 1$, so the group law is the additive one, as desired.
\end{proof}

\begin{proposition}
\label{WbigperftoA1perf}
For any $p$-complete animated ring $R$ and $R$-module $M$, the Teichmüller map $\Aoperf \to \Wbigperf$ induces an equivalence
\[ \RHom_{\Spf R}(\Wbigperf, \Ga \tensorhat_R M) \bij \RGamma((\Aoperf_{\Spf R}, 0), \Ga \tensorhat_R M) \mpunc. \]
\end{proposition}
\begin{proof}
By adjunction we may assume $R = \Zp$, and by $p$-completeness we may assume that $M$ is an $\Fp$-module, so by adjunction again we reduce to $R = \Fp$.
Taking the limit of the Postnikov tower, we reduce to $M$ coconnective.
The Breen--Deligne resolution \cite[Appendix to Lecture IV]{condensed} implies that $\Wbigperf$ is pseudocoherent, so we may commute out filtered colimits of coconnective modules $M$ and therefore assume $M = \Fp$.
Adjunction and perfectness of $\Wbigperf$ and $\Aoperf$ allow us to pass to the site of perfect $\Fp$-schemes, where the result follows by the argument of \cite[Theorem A.1]{akhilrh}.
\end{proof}

\begin{corollary}
\label{WbigperftoMmperf}
Let $\calR$ be a $1$-truncated $p$-complete ring in $p$-adic formal stacks over a ring $R$, and assume that $\Ga(\calR^\flat)$ is in the subcategory of abelian group stacks over $\Spf R$ generated under limits by $\bfB^n \Ga$ for $n \in \bbN$.
Then the Teichmüller map $\Mmperf \to \Wbigperf$ induces an equivalence
\[ \Map_{\Spf R}(\Wbigperf, \calR) \bij \Map_{\Spf R, *}(\Mmperf, \Mmperf(\calR)) \mpunc. \]
\end{corollary}
\begin{proof}
By \Cref{perfdeltamapout} and \Cref{Mmperfvsflat} we may replace $\calR$ with $\calR^\flat$, and by truncatedness we may work in the $\Einfty$ setting.
The result then follows from \Cref{WbigperftoA1perf}, \Cref{mapobjcalg} (applied to $\bbZ[\Mmperf]/([0]) \to \Wbigperf$), and the tensor-$\Hom$ adjunction.
\end{proof}

\subsection{From \texorpdfstring{$\Wperf$}{W{\textasciicircum}perf} to \texorpdfstring{$W$}{W}}
\label{sec:WperftoW}

In this section, we use the results of \Cref{sec:LGamma} to give a criterion for $\Wperf$-algebra structures on ring stacks to factor through $W$.

\begin{proposition}
\label{affstkWbigmap}
Let $M \to \Wperf$ be a flat descendable affine quasi-ideal scheme over a $p$-nilpotent base such that $\calR \defeq \Wperf/M$ is a $p$-complete ring in affine stacks (see \Cref{sec:affstk}), and such that $\Ga(\calR^\flat)$ lies in the subcategory of abelian group stacks generated under limits by $\bfB^n \Ga$ for $n \in \bbN$.\todo*[may not need this assumption]
Suppose that $\Wperf \to \calR$ fits into a commutative diagram
\[\xymatrix{
  \Mmperf \ar[d] \ar[r] & \Wperf \ar[d] \\
  \Mm \ar[r] & \calR
}\]
of commutative monoids, where the upper horizontal and left vertical maps are the natural ones.
Then the lower horizontal map extends naturally to a map of ring stacks
\[ W \to \calR \]
factoring $\Wperf \to \calR$.
If $\calR$ is static, then this extension is unique.
\end{proposition}
\begin{proof}
Note that the nerve of $\Wperf \to \calR$ is a simplicial flat descendable affine ring scheme, so we obtain from \Cref{Ntrdef}, \Cref{rmk:Ntrmonoid}, and \Cref{rmk:affstkpdesc} a map $\Ntr(\calR, 0) \to \calR$ of $\Einfty$-rigs in pseudodescendable sheaves.
\Cref{rmk:Wratplus} yields a map $\Wratplus \to \Ntr(\calR, 0)$, so the composite is a map $\Wratplus \to \calR$ of $\Einfty$-rigs in pseudodescendable sheaves, hence of $\Einfty$-rig stacks.
By \Cref{affringstkWratW} this map extends uniquely to a map $\Wbig \to \calR$ of $\Einfty$-rings in commutative group stacks, which by $1$-truncatedness is in fact a map of ring stacks.

We then obtain a commutative diagram
\[\xymatrix{
  \Wbigperf \ar[d] \ar[r] & \Wperf \ar[d] \\
  \Wbig \ar[r] & \calR
}\]
from \Cref{WbigperftoMmperf}.
Precomposing $\Wbigperf \to \calR$ with $\Wperf \to \Wbigperf$ then yields our original map $\Wperf \to \calR$, so precomposing $\Wbig \to \calR$ with $W \to \Wbig$ yields a factorization of this map through $W$.
Our map $W \to \calR$ thus factors $\Mmperf \to \calR$, so by \Cref{MmtoMmperfff} and $1$-truncatedness it factors $\Mm \to \calR$ as well.

The uniqueness statement follows from the universal property of $\Ntr$ when mapping to affine monoid schemes, along with the uniqueness statements for the other results used.
\end{proof}

In fact, if the map $W \to \calR$ produced above ends up being the quotient by a polyfiltered Cartier--Witt divisor, then it is \emph{a posteriori} unique.\footnotemark
\footnotetext{This uniqueness is implicitly used in the proof of \cite[Proposition 8.9.3(i)]{prismatization}.}
This follows from the fact, proved below, that the kernel of $\Wperf \to W$ has no nontrivial maps to $F_*^n W$ or $F_*^n \Gasharp$ for any $n$.

\begin{lemma}
\label{kerWperfW}
The $\Wperf$-module $K \defeq \ker(\Wperf \to W)$ admits a decreasing $\bbN$-indexed filtration whose $n$\textsuperscript{th} graded piece is given by $F_*^{-n-1} \Gasharp$.
\end{lemma}
\begin{proof}
We have $\Wperf = \lim_F F_*^{-n} W$, so $K = \lim \ker(F_*^{-n} W \to W)$.
As the kernel of $F_*^{-n} W \to F_*^{-n+1} W$ identifies with $F_*^{-n} \Gasharp$, we conclude.
\end{proof}

\begin{corollary}
Over a $p$-local base, $\Hom_{\Wperf}(K, F_*^n W)$ and $\Hom_{\Wperf}(K, F_*^n \Gasharp)$ vanish for all $n$.
\end{corollary}
\begin{proof}
As $F_*^n \Gasharp$ injects into $F_*^n W$ it suffices to prove the statement for the latter.
The $V$-adic filtration on $F_*^n W$ has graded pieces which are Frobenius twists of $\Ga$, so it suffices to see that $\Hom_{\Wperf}(K, F_*^n \Ga)$ vanishes for all $n$.
For this we may forget to $\Mmperf$-equivariant homomorphisms, where a map to $F_*^n \Ga$ is given by a primitive element in weight $p^n$ of the coordinate ring of the source (for the $\bbN[\tfrac{1}{p}]$-grading induced by the action of $\Mmperf$).
The choice of such an element commutes with filtered colimits along injective maps of Hopf algebras, so by \Cref{kerWperfW} we reduce to proving vanishing of $\Hom_{\Mmperf}(F_*^{-m} \Gasharp, F_*^n \Ga)$ for all $m, n \in \bbN$ with $m > 0$.
Taking graded Cartier duals, this amounts to vanishing of $\Hom_{\Mmperf}(F_*^n \Gasharp, F_*^{-m} \Ga)$.
The graded coordinate ring of the source is concentrated in integer weights, while that of the target is generated in weight $p^{-m}$, so any map must be trivial.
\end{proof}

The proof of \Cref{affstkWbigmap} used the following lemmas.

\begin{lemma}
\label{affringstkWratW}
If $\calR$ is an $\Einfty$-ring in commutative group affine stacks, then the natural map
\[ \Map(\Wbig, \calR) \to \Map(\Wrat, \calR) \]
is an equivalence.
\end{lemma}
\begin{proof}
This follows from \Cref{SMmWratW}, \Cref{mapobjcalg}, and the tensor-$\Hom$ adjunction.
\end{proof}

\begin{lemma}
\label{MmtoMmperfff}
Restriction of scalars from $\bbZ[\Mm]$ to $\bbZ[\Mmperf]$ induces a fully faithful functor from $\Einfty$-ring affine stacks under the former to those under the latter.
The essential image of this functor consists of those ring affine stacks whose $\bbN[\tfrac{1}{p}]$-graded coordinate ring vanishes in weights outside of $\bbN$.
\end{lemma}
\begin{proof}
Clearly this functor lands in the specified subcategory.
To see that it induces an equivalence, note first that the natural map of $\bbZ[\Mmperf]$-module mapping anima
\[ \intHom_{\bbZ[\Mmperf]}(\bbZ[\bbA^1], \calR) \to \intHom_{\bbZ[\Mmperf]}(\bbZ[\Aoperf], \calR) \]
identifies with the map
\[ \intMap_{\Mmperf}(\bbA^1, \calR) \to \intMap_{\Mmperf}(\Aoperf, \calR) \mpunc, \]
which is an equivalence by \Cref{gradedDAlgwelldefined} and \Cref{gradedDAlgrightadjoint}.
The desired statement now follows from \Cref{mapobjcalg} and the tensor-$\Hom$ adjunction.
\end{proof}

\subsection{The main result}
\label{sec:synmainresult}

The following result is the main input to the full faithfulness statement of \Cref{ZpSynmainthm}.

\begin{proposition}
\label{Wperffactorization}
For a $p$-nilpotent ring $R$ over $\ZpN$, any $\Wperf$-algebra structure on the ring stack $(\Ga^\Nyg)_R$ factors uniquely through the canonical one.
\end{proposition}
\begin{proof}
We first check the statement for $\Wbigperf$, where the canonical $\Wbigperf$-algebra structure on $(\Ga^\Nyg)_R$ is obtained by restricting along $\Wbigperf \to \Wperf$.
Here, by \Cref{WbigperftoMmperf}, it suffices to show that the map
\[ \Map_R(\Mmperf, \Mmperf) \to \Map_R(\Mmperf, \Mm^\Nyg) \]
is an equivalence.
Now by perfectness we may replace the target with $(\Mm^\Nyg)^\perf$,\footnotemark{} which is equivalent to $\Mmperf[][\Nyg]$ by the natural map.
\footnotetext{
  It is not clear to the authors whether the ``left perfection'' of an $\Einfty$-monoid is always computed as the limit along the $p$\textsuperscript{th}-power map, but whenever the latter is perfect it is indeed the perfection.
  This in particular holds for any $\Einfty$-monoid admitting an abelian monoid structure.
  See \cite{monoidperfection} for the dual case.
}
But by \Cref{perfdeltaprismatization} the map $\Mmperf \to \Mmperf[][\Nyg]$ is an equivalence, yielding the result.

The statement for $\Wperf$ now follows from the fact that it is a retract of $\Wbigperf$: given a map $\Wperf \to (\Ga^\Nyg)_R$, we obtain from the above a factorization of the composite $\Wperf \to \Wbigperf \to \Wperf \to (\Ga^\Nyg)_R$.
Uniqueness follows similarly.
\end{proof}

\begin{remark}
\label{rmk:ZpSynff2}
One can give a slightly different proof of \Cref{Wperffactorization} using \Cref{WmodMflat} and \Cref{perfdeltamapout}.
\end{remark}

We will also need the following result.

\begin{proposition}
\label{MmNpcompletenessautomatic}
Let $R$ be a $p$-nilpotent ring over $\ZpN$.
Then any ring structure on the monoid stack $(\Mm^\Nyg)_R$ is $p$-complete.
\end{proposition}
\begin{proof}
Suppose we have such a ring structure $\calR$.
By \Cref{MmMmNflatness} and accessibility of $\Mm$ and $\Mm^\Nyg$, any $R$-algebra admits a flat hypercover by $R$-algebras $S$ for which $\Mm(S) \to \Mm^\Nyg(S)$ is surjective on $\pi_0$.
Since $\Mm^\Nyg$ is a hypersheaf, and since $p$-completeness is preserved under limits, it suffices to check for such $S$.

For any minimal prime of $\calR(S)$, we may lift the corresponding multiplicative subset along $\Mm(S) \to \calR(S)$ to obtain a multiplicative subset of $S$, localization at which is necessarily nonzero.
Note that evaluating $\calR$ at this localization of $S$ yields the corresponding localization of $\calR(S)$: this is equivalent to the corresponding fact for $\Ga^\Nyg$, which follows from \cite[Remark 3.9]{blprismatization} along with the description of $\Ga^\Nyg$ as a pullback as in \cite[Definition 6.4.5]{moduliofFgauges}.
As $p$-nilpotence may be checked after localization at each minimal prime, we may assume by localizing on $S$ that $\calR(S)$ has a unique point.

Let $S \to k$ be a map to a perfect field.
Then $\Mm(\calR)$ is given on perfect $k$-algebras by $\Mmperf$, so by \Cref{Mmperfringuniqueperf} we have $\calR(k) \isom k$.
Thus the unique point of $\calR(S)$ is of characteristic $p$, as desired.
\end{proof}

\begin{theorem}
\label{ZpSynmainthm}
For a ring $A$, the natural map of c-stacks
\[ A^\Syn \to \AlgStk{A} \]
is fully faithful.
Furthermore, given an $A$-algebra stack $\calR$ over a $p$-nilpotent ring $R$, the following are equivalent: \begin{enumerate}
\item
  \label{ZpSynmainthm:taut}
  $\calR$ comes from an $R$-point of $A^\Syn$;
\item
  \label{ZpSynmainthm:def}
  $\calR$ is locally given by $W/M$ for some flat affine quasi-ideal scheme $M \to W$ which is a polyfiltered Cartier--Witt divisor modulo a finitely-generated nilpotent ideal of $R$;
\item
  \label{ZpSynmainthm:monoid}
  $\Mm(\calR)$ comes locally from an $R$-point of $\ZpSyn$.
\end{enumerate}
\end{theorem}
\begin{proof}
The first statement follows from \Cref{Wperffactorization}, \cite[Theorem 8.4.7]{prismatization}, and \Cref{transff}.
For the second, it is clear that (\ref{ZpSynmainthm:taut}) implies (\ref{ZpSynmainthm:def}) and (\ref{ZpSynmainthm:monoid}), and (\ref{ZpSynmainthm:def}) implies (\ref{ZpSynmainthm:taut}) by \Cref{passableWmodstrat} and \Cref{vfCWdivstrat}.
It remains to see that (\ref{ZpSynmainthm:monoid}) implies (\ref{ZpSynmainthm:taut}).


Localizing on $R$, we may assume that $\Mm(\calR)$ comes from an $R$-point of the stack $\QdRNyg$, so in particular it receives a map $\Mm \inj \Mm(W) \to \Mm(\calR)$.
By \Cref{MmNpcompletenessautomatic}, \Cref{WmodMflat}, \Cref{Mmperfringunique}, and \Cref{perfdeltamapout} we obtain a diagram of commutative monoids
\[\xymatrix{
  \Mmperf \ar[d] \ar[r] & \Wperf \ar[d] \\
  \Mm \ar[r] & \calR
}\]
where the right vertical map has the structure of a ring map.

Now assume further that $\Spec R \to \QdRNyg$ factors through the Nygaard-complete locus $\QdRNyghat$.
Then on $R / (p, q^{1/p} - 1, t)$ the map factors through $\Fp^\dRc$.
Thus on this locus, by \cite[Construction 2.7.11]{fgauges} and \Cref{Mmringunique}, we have $\pi_0(\calR) \isom \Ga$ and $\pi_1(\calR) \isom L^\sharp$ for some line bundle $L$ on $R / (p, q^{1/p} - 1, t)$.
Note that $\ker(\Wperf \to \calR)$ is then an extension of $\ker(\Wperf \to \Ga)$ by $L^\sharp$, so, as each of these is flat, it is as well.

If $\Spec R \to \QdRNyg$ instead factors through the complement
\[ \jdR(\Spf \QdRtw) \mpunc, \]
we can in fact apply the same argument because of the equivalence
\[ \jdR(\Spf \QdRtw) \isom \jHT(\Spf \QdR) \mpunc, \]
noting that the induced maps from $\Wperf$ and $\Mm$ will differ from the given ones by a Frobenius untwist.
Thus $\Wperf \to \calR$ is flat on $R [\tfrac{1}{t}] / (p, q^{1/p} - 1)$ as well, so by \Cref{stackflatstrat} we conclude that this map is flat on $R$ itself.

By \Cref{rmk:MmMmNflatdescendable} this map is also descendable, so by \Cref{GaSynaffineness} and \Cref{affstkWbigmap} we find that the map $\Mm \to \Mm(\calR)$ extends to a ring map $W \to \calR$.
By the naturality and uniqueness statement of \Cref{affstkWbigmap} this map, restricted to the mod-$t$ locus of $R / (p, q^{1/p} - 1)$, induces on $\pi_0$ the usual map $W \to \Ga$, so it is the quotient by a filtered Cartier--Witt divisor.

On the $t$-inverted locus, we can apply the same argument after Frobenius-untwisting as above to obtain a map $W \to \calR$ which is the quotient by a filtered Cartier--Witt divisor.
It follows from the construction of the map in \Cref{affstkWbigmap} that our original map $W \to \calR$ is the Frobenius twist of this one.
As such, it is the quotient by a $2$-polyfiltered Cartier--Witt divisor.

Finally, we know by \Cref{MmMmNkerfiniteness} and \Cref{gradedpolyquotient} that the kernel of $W \to \calR$ is weightwise finite locally free over all of $R$, so \Cref{vfCWdivstrat} allows us to conclude.
\end{proof}

\begin{remark}
By virtue of \Cref{ZpSynmainthm}, one can view $\ZpSyn$ as the moduli stack of ring structures on a certain monoid stack (see also \Cref{conj:ZpSynSSyn}).
\end{remark}

\begin{remark}
We would like to have a ``pointwise'' description of the essential image in \Cref{ZpSynmainthm}, in terms of ring stacks satisfying some suitable conditions which at each field-valued point arise from $\Fp^\Syn / p$.
Some form of such a description can be obtained from the methods used to prove the \namecref{ZpSynmainthm}, but the conditions most naturally required are somewhat cumbersome and thus unsatisfactory.
\todo*[would be nice to resolve this]
\end{remark}

\begin{remark}
In the situation of \Cref{ZpSynmainthm}(\ref{ZpSynmainthm:monoid}), we are given a ring stack $\calR$ over a ring $R$ together with a map $\Mm \to \calR$ of monoids with $0$.
This yields not just a cohomology theory associated to $\calR$ but also its variants relative to monoid rings.
Indeed, suppose we have an affine scheme $X$ over $A \defeq \bbZ[M]/([0])$, where $M$ is a monoid with $0$.
We can then define the relative transmutation as the pullback
\[ \prn*{X / A}^\calR \defeq X^\calR \times_{A^\calR} \Spec (A \tensor R) \mpunc, \]
where $(-)^\calR$ denotes the usual transmutation by $\calR$ and the right map is obtained from the map $\Mm \to \calR$ above.

From this point of view, one of the main difficulties in the proof of \Cref{ZpSynmainthm} is to extend this to allow $A$ to be a $\delta$-ring: this exactly corresponds to the extension from $\Mm$ to $W$.
\end{remark}

\begin{remark}
\label{rmk:qFdR}
The work \cite{evenstack} (see also \cite[\S II.7]{sanaththesis}) constructs a monoid stack $\Mm^\qFdR$ over the stack $\ZpqFdR \defeq (\Spf \Zp\Brk{\beta}\ang{t}) / \Gm$, transmutation by which computes ``$q$-Hodge filtered $q$-de Rham cohomology'' of monoid rings.\footnotemark
\footnotetext{It seems likely that this stack can also be constructed using the ideas of \cite{pridhamqdR}.}
Further, the pullback of this monoid stack along the (surjective) map
\[ \QdRNyg \to \ZpqFdR \]
sending $t \mapsto t$ and $\beta \mapsto (q^{1/p} - 1) u$ yields the stack $(\Mm[\Zpzetap] / \QdR)^\Nyg$.
\Cref{ZpSynmainthm} therefore implies that $\ZpSyn$ is the stack of ring stacks with underlying monoid stack locally equivalent to $\Mm^\qFdR$.
Thus any sufficiently-structured refinement of $q$-Hodge filtered $q$-de Rham cohomology of monoid rings to a cohomology theory for all rings must come from a point of $\ZpSyn$.
\end{remark}

\begin{remark}
We do not know a moduli description of the map
\[ \QdRNyg \to \ZpqFdR \]
along the lines of \Cref{ZpSynmainthm}, but we expect that the source is \emph{not} simply the stack of ring structures on $\Mm^\qFdR$.
\end{remark}

\begin{remark}[Thick prismatization]
\label{rmk:shearedprismatization}
The work \cite{shearedprismatization} defines a certain ``thickening'' of the prismatization modeled on the difference between the ``unipotent'' de Rham stack $\Ga/\Gasharp$ and the true de Rham stack $\Ga/\Gasharphat$.
The authors of \emph{op.\ cit.\@} prove the analogue of \Cref{WmodMflat} for their ring stack $\Ga^\Prismtilde$, so we expect the argument of \Cref{ZpSynmainthm} to similarly show full faithfulness of the map
\[ A^\Prismtilde / \phi^\bbN \to \AlgStk{A} \]
of c-stacks, where $(-)^\Prismtilde$ is the functor of \emph{op.\ cit.}

Identifying the essential image, on the other hand, seems more difficult, as in \Cref{rmk:QFdRmonoidstk}.
\end{remark}

\foralltheorems{
  \numberwithin{#1}{section}
  \numbernotwithin{#1}{subsection}
}

\section{De Rham}

Here we will show that the filtered de Rham stack of an animated $\bbQ$-algebra $R$ embeds fully faithfully into the stack of $R$-algebra stacks.
Our strategy for doing so is to use some $\RHom$ computations in abelian sheaves to reduce from maps of ring stacks to maps of generalized Cartier divisors.

In this section, all stacks are derived, but we restrict to the pro-fppf topology rather than the fpqc topology, as we do not know fpqc acyclicity of $\Gahat$ on affines.

The following definition is inspired by \cite[\S 6.4]{moduliofFgauges}.

\begin{definition}
Write $\QFdR$ for the c-stack $(\AomodMm)_\bbQ$.
We have a virtual Cartier divisor $\Ga \to \Gabar$ in ring stacks over $\QFdR$ (see \cite[\S 3.2]{gcds}), and we define the $\bbQ$-algebra stack $\Ga^\FdR \to \QFdR$ as the pullback
\[\xymatrix{
  \Ga^\FdR \pullbackcorner \ar[d] \ar[r] & \Ga^\dR \ar[d] \\
  \Gabar \ar[r] & \Gabar \tensor_{\Ga} \Ga^\dR
  \mpunc,
}\]
where $\Ga^\dR$ is the usual de Rham stack of $\Ga$, given by $\Ga^\dR(R) \defeq \pi_0(R)_\red$.
Note that the restriction of $\Ga^\FdR$ to $1 \in \QFdR$ is simply $\Ga^\dR$.

For any animated $\bbQ$-algebra $R$, we now define the c-stack $R^\FdR$ as the transmutation of $R$ by $\Ga^\FdR$, and we define $R^\dR$ and $R^\Hodge$ as its fibers over $1$ and $0$, respectively, in $\QFdR$.\footnotemark
\footnotetext{%
  Note that this does not generally give the correct de Rham stack for infinite-type schemes; see \cite[\S 7.1]{stacksatiyahsegal1} for a discussion of this point.
  One can obtain the correct transmutation by regarding $\Ga^\FdR$ as a stack of ind-$\bbQ$-algebras; it is then natural to ask for full faithfulness of the map from $\QFdR$ to the stack of ind-$\bbQ$-algebra stacks, but we do not pursue this here.\todo*[it might be easy, but it's unclear]
}
\end{definition}

The use of the word ``stack'' above to describe $\Ga^\FdR$ is justified by the following \namecref{GahatsheafQ}.

\begin{proposition}
\label{GahatsheafQ}
The assignment $R \mapsto \Gahat(R)$ defines a $\calD(\bbQ)$-valued pro-fppf hypersheaf on animated $\bbQ$-algebras.
Thus $R \mapsto \Ga^\dR(R) = R_\red$ does as well.
\end{proposition}
\begin{proof}
The statement for $\tau^{\leq -1} \Gahat$ follows from descent for modules, so it suffices to show the statement in degree $0$.
By staticness and finitarity of $\Gahat$ it suffices to prove fppf descent, which is due to de Jong \cite[Remark 2.2.18]{fgauges}.
\end{proof}

\begin{proposition}[folklore]
\label{QRHomGaGabase}
The natural map
\[ \RHom_\bbQ(\Ga, \Ga) \to \RHom_\bbQ(\bbZ, \Ga) \isom \bbQ \]
is an equivalence.
\end{proposition}
\begin{proof}
The statement in degree $0$ is easy; for the full statement it remains to show that the left-hand side has no higher cohomology.
For this we may instead consider
\[ \RHom_\bbQ(\Wbig, \Ga) \mpunc, \]
as in characteristic $0$ the group scheme $\Wbig$ is an infinite product of copies of $\Ga$, and the relevant functor commutes with cofiltered limits of affine group schemes.
By \cite[Proposition A.6]{akhilrh} we may further replace $\Wbig$ by the rational Witt vectors $\Wrat$.
As $h$- and pro-fppf cohomology of $\Ga$ on smooth schemes in characteristic $0$ agree by \cite[Corollary 6.5]{hdiffforms}, and as $\bbA^1$ and $\Wratplus$ are ind-smooth, we may pass to the $h$-topology.
We then conclude by the argument of \cite[Proposition A.15]{akhilrh} that the natural map
\[ \RHom_\bbQ(\Wrat, \Ga) \to \RGamma((\bbA^1, 0), \Ga) \]
is an equivalence, whence the result follows.\footnotemark
\footnotetext{
  One can give a slightly different version of the last step as follows.
  The Breen--Deligne resolution implies that it suffices to check the analogous statement for $\Hom_R(\Wrat, \Ga)$, where $R = \bbQ[\epsilon]/\epsilon^2$ with $\epsilon$ in degree $n$, for each $n > 0$.
  This statement then follows from the fact that $\Wratplus$ is the free commutative monoid ind-(affine derived scheme) on $(\bbA^1, 0)$ in characteristic $0$.
}
\end{proof}

\begin{corollary}
\label{QRHomGaGa}
For an animated $\bbQ$-algebra $R$ and $R$-module $M$, the natural map
\[ \RHom_R(\Ga, \Ga \tensor M) \to \RHom_R(\bbZ, \Ga \tensor M) \isom M \]
is an equivalence.
\end{corollary}
\begin{proof}
As in \Cref{WbigperftoA1perf}, we may reduce to the case $R = M = \bbQ$, which is \Cref{QRHomGaGabase}.
\end{proof}

\begin{lemma}
\label{Gareduced}
For an animated ring $R$, the natural map
\[ \calO(\bbA^n_R)_\red \to \calO(\bbA^n_{R_\red}) \]
is an isomorphism.
\end{lemma}
\begin{proof}
It suffices by induction to prove the statement for $n = 1$.
Both sides send $R \to \pi_0(R)$ to an equivalence, so it suffices to check the statement for $R$ static.
In this case, we reduce to the claim that every nilpotent element of $R[t]$ is given by a polynomial with nilpotent coefficients.
To see this, let $f$ be such an element.
Consider the highest degree term $a_i t^i$ of $f$, and note that $f^n$ has $t^{i n}$-coefficient $a_i^n$.
Since $f$ is nilpotent, $a_i$ must be as well, so $f - a_i t^i$ is nilpotent and of degree lower than that of $f$.
We conclude by induction.
\end{proof}

\begin{corollary}
\label{QRHomGaGadR}
For an animated $\bbQ$-algebra $R$, the natural map
\[ \RHom_R(\Ga, \Ga^\dR) \to \RHom_R(\bbZ, \Ga^\dR) \isom R_\red \]
is an equivalence.
\end{corollary}
\begin{proof}
By \Cref{Gareduced}, the Breen--Deligne spectral sequences computing the objects $\RHom_R(\Ga, \Ga^\dR)$ and $\RHom_{R_\red}(\Ga, \Ga)$ are isomorphic by the natural map, so we obtain
\[ \RHom_R(\Ga, \Ga^\dR) \isom \RHom_{R_\red}(\Ga, \Ga) \isom R_\red \mpunc. \qedhere \]
\end{proof}

\begin{lemma}
For an animated $\bbQ$-algebra $R$, the objects
\[ \RHom_R(\Ga^\dR, \Ga) \quad \textrm{and} \quad \RHom_R(\Gahat, \Ga^\dR) \]
vanish.
\end{lemma}
\begin{proof}
As the natural maps $R \to \RGamma((\Ga^\dR)^n)$ and $R \to \RGamma((\Gahat)^n, \Ga^\dR)$ are equivalences for all $n$, the $\bbE_1$ pages of the Breen--Deligne spectral sequences computing the objects in question agree naturally with that computing $\RHom_R(0, \Ga)$.
The result follows immediately.
\end{proof}

\begin{corollary}
\label{QRHomGaGahatetc}
For an animated $\bbQ$-algebra $R$, the natural maps
\begin{align*}
  \RHom_R(\Ga, \Gahat) &\to \RHom_R(\bbZ, \Gahat) \isom \Gahat(R) \mpunc, \\
  \RHom_R(\Ga, \Ga) &\to \RHom_R(\Gahat, \Ga) \mpunc{, and} \\
  \RHom_R(\Gahat, \Gahat) &\to \RHom_R(\Gahat, \Ga)
\end{align*}
are equivalences.
\end{corollary}
\begin{proof}
This follows from the results above together with the fiber sequence
\[ \Gahat \to \Ga \to \Ga^\dR \mpunc. \qedhere \]
\end{proof}

\begin{lemma}
\label{QEndGa}
For an animated $\bbQ$-algebra $R$, the group
\[ \Map_{\Sh(R; \aCAlg_\bbQ)}(\Ga, \Ga) \]
is trivial.
\end{lemma}
\begin{proof}
By \Cref{mapobjcalg} it suffices to see for each $n$ that the natural map
\[ \Hom_R(\Ga^{\tensor n}, \Ga) \to \Hom_R(\bbQ^{\tensor n}, \Ga) \]
is an equivalence, which follows from \Cref{QRHomGaGa} and the tensor-$\Hom$ adjunction.
\end{proof}

\begin{proposition}
\label{QFdRGaunique}
Suppose $R$ is an animated $\bbQ$-algebra over $\QFdR$.
Then the $\Ga$-algebra structure on $(\Ga^\FdR)_R$ is unique.
\end{proposition}
\begin{proof}
As above, by \Cref{mapobjcalg} it suffices to see for each $n$ that the natural map
\[ \RHom_R(\Ga^{\otimes n}, (\Ga^\FdR)_R) \to \RHom_R(\bbZ^{\otimes n}, (\Ga^\FdR)_R) \]
is an equivalence, which follows from \Cref{QRHomGaGahatetc} and the tensor-$\Hom$ adjunction, as $(\Ga^\FdR)_R$ is locally a quotient of $\Ga$ by $\Gahat$.
\end{proof}

Below, for an animated $\bbQ$-algebra $R$ equipped with a generalized Cartier divisor $s : L \to R$, we write $\Ga \slashsub{s} L$ for the tensor product of the ring stack $\Ga[R]$ over $R$ with the virtual Cartier divisor associated to $s$, and we write $\Ga \slashsub{s} \widehat{L}$ for the pullback of $\Ga^\FdR$ along the induced $R$-point of $\QFdR$.

\begin{proposition}
\label{QMapdRvsgCd}
For an animated $\bbQ$-algebra $R$ with generalized Cartier divisors $s : L \to R$ and $t : L' \to R$, the natural maps
\begin{align*}
  \Map_{\Ga}(\Ga \slashsub{s} \widehat{L}, \Ga \slashsub{t} \widehat{L}')
    &\to \Map_{\Ga}(\Ga \slashsub{s} \widehat{L}, \Ga \slashsub{t} L')
  \\&\from \Map_{\Ga}(\Ga \slashsub{s} L, \Ga \slashsub{t} L')
\end{align*}
are equivalences.
\end{proposition}
\begin{proof}
By \Cref{mapobjcalg} (applied to $\CAlg(\Sh(R; \bbQ)^{\Delta^1})$) and \Cref{QEndGa} it is enough to see that the natural maps
\begin{align*}
  & \Hom_R((\Ga \slashsub{s} \widehat{L})^{\tensor n}, \Ga \slashsub{t} \widehat{L}') \times_{\Hom_R(\Ga^{\tensor n}, \Ga \slashsub{t} \widehat{L}')} \Hom_R(\Ga^{\tensor n}, \Ga) \\
  & \quad \to \Hom_R((\Ga \slashsub{s} \widehat{L})^{\tensor n}, \Ga \slashsub{t} L') \times_{\Hom_R(\Ga^{\tensor n}, \Ga \slashsub{t} L')} \Hom_R(\Ga^{\tensor n}, \Ga) \\
  \shortintertext{\centering and}
  & \Hom_R((\Ga \slashsub{s} L)^{\tensor n}, \Ga \slashsub{t} L') \times_{\Hom_R(\Ga^{\tensor n}, \Ga \slashsub{t} L')} \Hom_R(\Ga^{\tensor n}, \Ga) \\
  & \quad \to \Hom_R((\Ga \slashsub{s} \widehat{L})^{\tensor n}, \Ga \slashsub{t} L') \times_{\Hom_R(\Ga^{\tensor n}, \Ga \slashsub{t} L')} \Hom_R(\Ga^{\tensor n}, \Ga)
\end{align*}
are equivalences for all $n$, which follows from \Cref{QRHomGaGahatetc}.
\end{proof}

\begin{corollary}
\label{QFdRmainthm}
For an animated $\bbQ$-algebra $R$, the natural map from the c-stack $R^\FdR$ to the stack of $R$-algebra stacks is fully faithful.
\end{corollary}
\begin{proof}
By \Cref{transff} we may assume $R = \bbQ$.
The statement then follows from \Cref{QMapdRvsgCd}, \Cref{QFdRGaunique}, and (the proof of) \cite[Proposition 3.2.6]{gcds}.
\end{proof}

\begin{remark}
\label{rmk:QFdRmonoidstk}
We do not know whether the essential image in \Cref{QFdRmainthm} is determined by the underlying monoid stacks, although it is easy to see that it is on the de Rham locus.
\end{remark}

\begin{remark}[The de Rham stack in mixed characteristic]
\label{rmk:dRnotff}
We expect that the de Rham stack in mixed characteristic is \emph{not} fully faithful in ring stacks: for the unipotent de Rham stack this follows from the existence of extra automorphisms of de Rham cohomology, as in \cite[Remark 4.7.18]{apc} and \cite{enddR}, and for the true de Rham stack it would follow from an analogue thereof in the context of \cite{shearedprismatization}.
However, in characteristic $p$ the only failure of full faithfulness should be the Frobenius endomorphism; see \Cref{rmk:shearedprismatization}.
\end{remark}

\section{Étale}

\begin{definition}
\label{petdef}
Given a $p$-complete ring $R$, we define its \emph{$p$-adic étale stack} as
\[ R^\pet \defeq \prn*{\Spf W((R/p)_\perf)} / \phi \mpunc. \]
\end{definition}

Note that the results of \cite{blrh} imply that sheaves on $R^\pet$ do in fact naturally fully faithfully contain $p$-adic étale sheaves on $\Spf R$.

\begin{warning}
\Cref{petdef} does not include any information about the generic fiber, and is therefore best applied to $p$-nilpotent rings.
We allow for $R$ not-necessarily-nilpotent in order to more easily describe $(-)^\pet$ by transmutation.
\end{warning}

\begin{theorem}
\label{petmainthm}
The $p$-adic étale stack is transmuted, and the natural map from $R^\pet$ to the stack of $R$-algebra stacks is fully faithful.
Any ring stack over a $p$-nilpotent ring with underlying monoid locally isomorphic to $\Mm^\pet$ comes from a point of $\Zp^\pet$.
\end{theorem}
\begin{proof}
Exercise.\footnotemark
\footnotetext{Hint: compute the ring endomorphisms of $\Gaflat$, then apply \Cref{translim}, \Cref{transff}, and \Cref{Mmperfringunique}.}
\end{proof}

\section{Betti}
\label{sec:betti}

Here we study the Betti stack construction, first as a functor from compactly generated Hausdorff spaces to stacks, then as a functor from schemes to stacks.
Throughout this section we fix a field $K$, assumed to be either $\bbR$ or $\bbC$.

We first recall the definition of the Betti stack, after \cite[Exercise 1.7]{6functors}.

\begin{definition}
Given a compactly generated Hausdorff space $X$, we define its \emph{Betti stack} as
\[ X^\rmB(R) \defeq \Map_{\Top}(\abspec{R}, X) \mpunc, \]
where $\Top$ is the category of topological spaces.

If $X$ is an affine scheme over a subring of $K$, we define its Betti stack as $X^\rmB \defeq X(K)^\rmB$.
\end{definition}

The Betti stack of a variety is in fact given by transmutation:
\begin{proposition}
For $X$ an affine scheme over $K$, there is a natural isomorphism
\[ X^\rmB(R) \isom X(\Ga^\rmB(R)) \mpunc. \]
\end{proposition}\begin{proof}
Note that the functor sending an affine scheme to its $K$-points (as a topological space) commutes with limits; in the case $K  = \bbC$ this is \cite[Proposition 5.4]{ksWrat}, and the case $K = \bbR$ follows similarly.
The result now follows from \Cref{translim}, as the Betti stack construction preserves limits of topological spaces.
\end{proof}

\begin{recollection}
\label{wlocal}
Recall that any ring admits a proétale cover by a ring which is \emph{w-local} in the sense of \cite[Definition 2.2.1]{proetale}.
By \cite[Lemma 2.1.4]{proetale}, any connected component of the spectrum of a w-local ring $R$ has a unique closed point; any continuous map from $\abspec{R}$ to a Hausdorff space thus factors through $\abspec{R} \to \pinspec{R}$.

Given a ring $R$, we write $\CAlgwl_R$ for the category of w-local $R$-algebras.
\end{recollection}

\begin{definition}
For a w-local ring $R$ and compactly generated Hausdorff space $X$ over $\pinspec{R}$, we define the \emph{relative Betti stack} of $X$ over $R$ as
\[ X^{\rmB/R}(S) \defeq \Map_{\Top_{/\pinspec{R}}}(\abspec{S}, X) \mpunc. \]
Note in the above that $\pinspec{R}$ is a profinite set by \citestacks{0906}.
\end{definition}

\begin{remark}
As in \cite[Proposition 2.20]{condensedstoneduality}, it follows from \cite[Proposition 1.6.2.4(2')]{sag} that the Betti stack and relative Betti stack are indeed sheaves.
The argument of \cite[Proposition 2.15]{condensed} shows that they are accessible.
\end{remark}

\begin{definition}
We write $\CGHaus$ for the category of compactly generated Hausdorff spaces, and we write $\Prof$ for the full subcategory of profinite sets.
\end{definition}

\begin{lemma}
\label{contfuncSpec}
For any profinite set $X$ and ring $R$, the natural map\todo*[what is the natural map, if not just defined by the limit over finite quotients of $X$ in the target?]
\[ \abspec{C(X, R)} \to X \times \abspec{R} \]
is an isomorphism, and the map $\Spec C(X, R) \to \Spec R$ induces isomorphisms on residue fields at each point.
\end{lemma}
\begin{proof}
By \citestacks{0CUF} the functor $\abspec{(-)}$ sends filtered colimits to limits, so it suffices for the first claim to consider the case of finite $X$, which is obvious.
For the second claim, note that the residue field of a point of $\Spec C(X, R)$ is the colimit of the residue fields of its images in $\Spec C(Y, R)$ for finite quotients $Y$ of $X$ by \citestacks{0CUG}, so we again reduce to the case of $X$ finite.
\end{proof}

\begin{lemma}
\label{contfuncffZ}
The functor
\[ C(-, \bbZ) : \Prof^\opp \to \CAlg \mpunc, \]
is fully faithful and has a right adjoint given by
\[ S \mapsto \pinspec{S} \mpunc. \]
\end{lemma}
\begin{proof}
Because $\Prof^\opp$ is generated under colimits by finite sets, which are compact, it suffices to verify the claims for finite sets and show that $C(X, \bbZ)$ is compact for finite $X$.

To this end, let $S$ be a ring and let $X$ be a finite set with $n$ elements.
Compactness of $C(X, \bbZ)$ follows from the fact that it is finitely presented.
A homomorphism of abelian groups $f : C(X, \bbZ) \to S$ is freely determined by its values on the indicator functions of points of $X$.
For $f$ to be a ring map, these values must be idempotent, multiply pairwise to 0, and sum to 1, so such $f$ correspond to partitions of $\abs{\Spec S}$ into $n$ clopen subsets.
These are exactly in correspondence with maps $\abs{\Spec S} \to X$ and therefore with maps $\pinspec{S} \to X$, proving the adjunction.

Because $C(-, \bbZ)$ takes coproducts of finite sets to products, it suffices to verify full faithfulness when the target is a singleton.
This follows from the unit of the adjunction being an isomorphism on singletons, since $\pinspec{\bbZ} \isom \{*\}$.
\end{proof}

\begin{corollary}
\label{contfuncff}
The functor
\[ R \otimes_{C(\pinspec{R}, \bbZ)} C(-, \bbZ) : \Prof_{/\pinspec{R}}^\opp \to \CAlg_R \]
is fully faithful with right adjoint given by
\[ S \mapsto \pinspec{S} \mpunc. \]
Further, if $R$ is w-local, then the functor above is valued in w-local $R$-algebras.
\end{corollary}
\begin{proof}
It follows from \Cref{contfuncffZ} that for any $R$-algebra $S$ there is a natural isomorphism
\[ \Map_{\Prof_{/\pinspec{R}}}(\pinspec{S}, X) \isom \Hom_{C(\pinspec{R}, \bbZ)}(C(X, \bbZ), S) \mpunc, \]
so the adjunction follows by base changing to $R$.

For full faithfulness it remains to verify that the unit map
\[ \pinspec{(R \otimes_{C(\pinspec{R}, \bbZ)} C(Y, \bbZ)} \to Y \]
is an isomorphism.
For this, first note that, as the map $C(\pinspec{R}, \bbZ) \to C(Y, \bbZ)$ induces isomorphisms on residue fields by \Cref{contfuncSpec}, the map
\begin{align*}
  \abspec{(R \otimes_{C(\pinspec{R}, \bbZ)} C(Y, \bbZ))}
    &\to \abspec{R} \times_{\abspec{C(\pinspec{R}, \bbZ})} \abspec{C(Y, \bbZ)}
  \\&\isom \abspec{R} \times_{\pinspec{R}} Y
\end{align*}
is an isomorphism by \citestacks{01JT}.
Thus, as the fibers of the map $\abspec{R} \to \pinspec{R}$ are connected, those of
\[ \abs{\Spec(R \otimes_{C(\pinspec{R}, \bbZ)} C(Y, \bbZ))} \to Y \]
are as well.
The claim now follows from the fact that $Y$ is totally disconnected.

To see w-locality, note that by the isomorphism above, \cite[Theorem 11.1.5]{spectralspaces}, and \cite[Lemma 2.1.9]{proetale} it suffices to show that the maps $\abspec{R} \to \pinspec{R}$ and $Y \to \pinspec{R}$ are w-local.
As any point of a profinite set is closed, it in fact suffices to see that these maps are spectral.
This follows from the fact that, as maps from quasicompact spaces to Hausdorff spaces, they are proper.
\end{proof}

\begin{proposition}
\label{bettiff}
For a w-local ring $R$, the functor sending a compactly generated Hausdorff space over $\pinspec{R}$ to its relative Betti stack is fully faithful.
\end{proposition}
\begin{proof}
Note that by \Cref{contfuncff} the left Kan extension of the functor
\[ R \otimes_{C(\pinspec{R}, \bbZ)} C(-, \bbZ) : \Prof_{/\pinspec{R}} \to (\CAlgwl_R)^\opp \]
to a functor
\[ \PShAcc(\Prof_{/\pinspec{R}}) \to \PShAcc((\CAlgwl_R)^\opp) \]
is fully faithful.
It follows from \cite[Proposition 1.7]{condensed} (and \cite[Proposition 2.15]{condensed}) that the restricted Yoneda functor
\[ \CGHaus_{/\pinspec{R}} \to \PShAcc(\Prof_{/\pinspec{R}}) \]
is also fully faithful.
Denote the image of $X \in \CGHaus_{/\pinspec{R}}$ under the composite of these two functors by $X^{\rmB'/R}$.
By \Cref{contfuncff} and the full faithfulness above, we have a natural identification
\[ X^{\rmB'/R}(S) \isom \Map_{\CGHaus_{/\pinspec{R}}}(\pinspec{S}, X) \mpunc. \]
Therefore, by the discussion in \Cref{wlocal}, $X^{\rmB'/R}$ is simply the restriction of $X^{\rmB/R}$ to w-local $R$-algebras.
Thus, as the forgetful functor from $0$-truncated $R$-stacks to $\PShAcc((\CAlgwl_R)^\opp)$ is fully faithful, the desired statement follows from full faithfulness of $(-)^{\rmB'/R}$.
\end{proof}

\begin{lemma}
\label{uniqfield}
There is a unique topological field structure on the topological space underlying $K$.
\end{lemma}
\begin{proof}
Let $X$ be an arbitrary topological field with underlying topological space homeomorphic to $K$.
It is classical that the underlying topological abelian group of $X$ is then isomorphic to that of $K$.\todo[cite something very old?]
We can therefore produce a homomorphism of topological abelian groups $f : \bbR \to X$ sending 1 to the multiplicative unit of $X$, which is an isomorphism onto its image.
By construction, this map will be a ring homomorphism after restriction to $\bbZ \subseteq \bbR$.
Furthermore, since inverses of integers in $X$ are determined by the additive structure, this is actually a ring homomorphism after restriction to $\bbQ \subseteq \bbR$.
By continuity, we therefore obtain that $f$ is a homomorphism of topological rings.

If $K = \bbR$, we are done because $f$ is an isomorphism onto its image.

If $K = \bbC$, note that $f$ exhibits $X$ as a 2-dimensional topological vector space over $\bbR$, since the underlying topological abelian group of $X$ is isomorphic to $\bbR^2$, and the action of $\bbR \subseteq X$ by multiplication is determined by its restriction to $\bbQ$.
Since $X$ is a field, we obtain immediately that it is isomorphic to $\bbC$ as a topological ring.
\end{proof}

\begin{lemma}
\label{monoidautR}
The map
\begin{align*}
  \bbR_{>0} &\to \Aut_{\CMon(\Top)}(\bbR) \\
  a &\mapsto (t \mapsto \abs{t}^{a-1} t)
\end{align*}
is an isomorphism.
\end{lemma}
\begin{proof}
Any monoid automorphism of $\bbR$ restricts to a group automorphism of the units $\bbR^\times \isom \bbR_{>0} \times \{\pm 1\}$.
Automorphisms of the latter group are given by sending $(t, s) \mapsto (t^a, s)$ for $a \neq 0 \in \bbR$, and the condition that this extends to a monoid automorphism on all of $\bbR$ forces $a > 0$.
\end{proof}

\begin{lemma}
\label{uniqringRS}
Let $S$ be a topological space, and let $X$ be a topological ring over $S$.
If the underlying topological monoid of $X$ is isomorphic to $S \times \bbR$, then $X$ is isomorphic to $S \times \bbR$ as a topological ring.
\end{lemma}
\begin{proof}
We can think of $X$ as an alternative addition on the topological monoid $S \times \bbR$ making it into a ring.
We will denote the addition from $X$ by $+_X$.
Consider the the map $2_X : S \to S \times \bbR$ given by $1 +_X 1$.
By \Cref{uniqfield} and \Cref{monoidautR}, the map $2_X$ lands in $S \times \bbR_{>0}$.
Let $\nu : \bbR_{>0} \times \bbR \to \bbR$ be defined by $\nu(x, t) = \abs{t}^{\log_2(a) - 1} t$.
Consider the composition
\begin{align*}
  S \times \bbR
    &\xrightarrow{\mathmakebox[3em]{2_X \times \id_\bbR}} S \times \bbR_{>0} \times \bbR
  \\&\xrightarrow{\mathmakebox[3em]{\id_S \times \nu}} S \times \bbR \mpunc.
\end{align*}
This is clearly a homeomorphism, and it follows from \Cref{uniqfield} and \Cref{monoidautR} that it sends the usual addition on $S \times \bbR$ to the addition from $X$, so we are done.
\end{proof}

\begin{lemma}
\label{monoidautC}
The map
\begin{align*}
  A : \bbR_{>0} \times \bbR \times \{\pm 1\} &\to \Aut_{\CMon(\Top)}(\bbC) \\
  (a, b, s) &\mapsto (e^{x + i y} \mapsto e^{a x + i (b x + s y)})
\end{align*}
is a homeomorphism, where the target is given the compact-open topology.
\todo[rewrite this to be well-defined at 0]
\end{lemma}
\begin{proof}
We have an isomorphism of topological groups $\bbR \times \bbR/2\pi\bbZ \to \bbC^\times$ from the exponential map, so that $\Aut_{\CMon(\Top)}(\bbC^\times)$ identifies with the subgroup
\[
\left\{
  \left.
  \begin{bmatrix}
    a & 0 \\
    b & s \\
  \end{bmatrix}
  \right\rvert
  a, b \in \bbR, s \in \{\pm 1\}, a \neq 0
\right\}
\subseteq
\GL_2(\bbR) \mpunc.
\]
Transporting this automorphism along the exponential map, we obtain the formula in the statement.
For one of these group automorphisms of $\bbC^\times$ to extend to a monoid automorphism on all of $\bbC$, we must have $a \in \bbR_{>0}$.
\end{proof}

\begin{lemma}
\label{twistaddC}
Consider the map
\begin{align*}
  B : \Aut_{\CMon(\Top)}(\bbC) &\to \Map(\bbC \times \bbC, \bbC) \\
  \sigma &\mapsto ((w, z) \mapsto \sigma(\sigma^{-1}(z) + \sigma^{-1}(w)))
\end{align*}
sending a monoid automorphism of $\bbC$ to addition ``twisted'' by the automorphism.
The composite
\[ \bbR_{>0} \times \bbR \times \{+1\} \xrightarrow{A} \Aut_{\CMon(\Top)}(\bbC) \xrightarrow{B} \Map(\bbC \times \bbC, \bbC) \]
is a homeomorphism onto its image, where $\Map(\bbC \times \bbC, \bbC)$ is given the compact-open topology. \todo*[indicate the restriction to +1 more clearly]
\end{lemma}
\begin{proof}
Let $X$ denote the image of this map with the subspace topology.
For the first component, consider the map $r_1 : X \to \bbR_{>0}$ defined by $r_1(x) = \log_2 \abs{x(1, 1)}$.
It follows immediately from the formula in \Cref{monoidautC} that we have $r_1(A(a, b, s)) = a$.
Because the map $B$ is equivariant for the action of $\Aut_{\CMon(\Top)}(\bbC)$ on $\Map(\bbC \times \bbC, \bbC)$ defined by
\begin{align*}
  \Aut_{\CMon(\Top)}(\bbC) \times \Map(\bbC \times \bbC, \bbC) &\to \Map(\bbC \times \bbC, \bbC) \\
  (\sigma, f) &\mapsto ((w, z) \mapsto \sigma(f(\sigma^{-1}(z), \sigma^{-1}(w)))) \mpunc,
\end{align*}
we reduce to showing that the composite
\[ \bbR \hookrightarrow \bbR_{>0} \times \bbR \times \{+1\} \xrightarrow{A} \Aut_{\CMon(\Top)}(\bbC) \xrightarrow{B} \Map(\bbC \times \bbC, \bbC) \]
is a homeomorphism onto its image. \todo*[better way of checking this]

Consider the map
\begin{align*}
    (0, 1] &\to \bbC \times \bbC \\
    t &\mapsto \left( \frac{t}{2} + i \sqrt{1-\frac{t^2}{4}}, \frac{t}{2} - i \sqrt{1-\frac{t^2}{4}} \right) \mpunc.
\end{align*}
Precomposing with this map, it suffices to show that the map
\begin{align*}
    \bbR &\to \Map((0, 1], \bbC) \\
    b &\mapsto (t \mapsto t e^{i b \log \abs{t}})
\end{align*}
is a homeomorphism onto its image.
Normalizing the image and precomposing with the exponential map, we finally reduce to checking that the map
\begin{align*}
    \bbR &\to \Map((-\infty, 0], \bbR/\bbZ) \\
    b &\mapsto (t \mapsto b t)
\end{align*}
is a homeomorphism onto its image, which can be verified explicitly. \todo[double-check]
\end{proof}

\begin{lemma}
\label{uniqringCS}
Let $S$ be a compactly generated Hausdorff topological space, and let $X$ be a topological ring over $S$.
If the underlying topological monoid of $X$ is isomorphic to $S \times \bbC$, then $X$ is isomorphic to $S \times \bbC$ as a topological ring.
\end{lemma}
\begin{proof}
As in the proof of \Cref{uniqringRS}, we will pull back the addition from $X$ to $S \times \bbC$ and denote it by $+_X$.
Consider the map $S \to \Map(\bbC \times \bbC, \bbC)$ adjoint to the addition $+_X$.
Because being a field is detected on underlying monoids, it follows from \Cref{uniqfield} that the image of this map lies in the image of the map from \Cref{twistaddC} and therefore lifts to a map $S \to \Aut_{\CMon(\Top)}(\bbC)$.
This map is then adjoint to a monoid automorphism of $S \times \bbC$ which by construction intertwines the usual addition with $+_X$, i.e.\ to a topological ring isomorphism from $S \times \bbC$ to $X$.
\end{proof}

\begin{theorem}
\label{Bettimainthm}
Let $R$ be an algebra over a subring $k \subseteq K$.
Then the map from $(R \tensor_k K)^\rmB$ to the stack of $R$-algebra stacks is fully faithful if $K = \bbR$ or if $k \subseteq \bbC$ is not contained in $\bbR$.
If $K = \bbC$ and $k$ is contained in $\bbR$, then this map exhibits its image as the quotient of its source by the natural action of $\Gal(\bbC/\bbR)$.

An $R$-algebra stack is in the essential image of $(R \tensor_k K)^\rmB$ if and only if its underlying monoid stack is.
\end{theorem}
\begin{proof}
Because the Betti stack is transmuted (as is the variant in the second statement above), it suffices by \Cref{transff} to show the claims when $R = k$, so we have $R^\rmB = \Spec \bbZ$.

The statement is then equivalent to the claim that, for any ring $S$, the anima of endomorphisms of $\Ga[S]^\rmB$ as a $k$-algebra stack over $S$ identifies via the natural map with
\[ \Gamma(\Spec S, \End_{\CAlg_k(\CGHaus)}(K)) \mpunc, \]
where $k$ is given the discrete topology and $K$ the usual topology.
Since $\Ga^\rmB$ and the stack of ring stacks are both fpqc sheaves, it suffices to verify the claim for $S$ w-local.
By \Cref{bettiff} together with the fact that the Betti stack preserves products, we have an isomorphism
\[ \End_{\CAlg_k(\Stk_S)}((\Ga^\rmB)_S) \isom \End_{\CAlg_{k \times \pinspec{S}}(\CGHaus_{/\pinspec{S}})}(K \times \pinspec{S}). \]
The desired statement is now clear.

The statement on the essential image follows immediately from \Cref{uniqringRS} and \Cref{uniqringCS}. \todo[double-check]
\end{proof}

\section{Complements}
\label{sec:complements}

\begin{warning}
\label{derivedcounterexample}
The analogue of \Cref{ZpSynmainthm} in the derived setting seems to be \emph{false}, at least if one takes a ring stack to simply be a sheaf of animated rings.
Indeed, we expect that one can give a counterexample as follows.

By \Cref{transff} it will suffice to see that the map of c-stacks
\[ \Fp^\Hodge / \phi^\bbN \to \AlgStk{\Fp} \]
is not fully faithful.\todo*[check quotient]
For this, it is enough to find an animated $\Fp$-algebra $R$ over which the natural map
\[ \underline{\bbN} \to \intEnd_{\Sh(R; \, \aCAlg_{\Fp})}(\Ga) \]
is not an equivalence.\todo*[double-check that this is enough]

Let $R \defeq \Fp[\epsilon] / \epsilon^2$, with $\epsilon$ in degree $3$, and assume $p \geq 5$, so we may compute in $\Einfty$- instead of animated rings.
We have by \Cref{BreenREndGafiltn} a fiber sequence
\[ \Ga{[F]} \to \intHom_{\Sh(R; \Fp)}(\Ga, \Ga) \to \bfB \pi_3 \Ga \]
when restricting to flat $R$-algebras.\footnotemark
\footnotetext{As $\Ga$ is flat, there is no harm in making this restriction.}
The tensor-$\Hom$ adjunction then yields, for any finite set $S$, a filtration on $\intHom_R(\Ga^{\tensor S}, \Ga)$ with associated graded given by
\[ \Ga{[\{F_i\}_{i \in S}]} \oplus \bfB \pi_3 \prn*{\Ga{[\{F_i\}_{i \in S}]} / \prn*{\textstyle\sum_{i \in S} F_i}}^{\oplus S} \mpunc. \]
\todo*[check naturality]
We expect that, computing algebra maps as in \Cref{mapobjcalg}, one can show that the second summand above prevents the map
\[ \underline{\bbN} \to \intEnd_{\Sh(R; \, \aCAlg_{\Fp})}(\Ga) \]
from being an equivalence.
\end{warning}

\begin{lemma}
\label{BreenREndGafiltn}
For $R$ an animated $\Fp$-algebra and $M$ an eventually-coconnective $R$-module, there is a natural increasing $\bbN$-indexed filtration of
\[ \RHom_{\Sh(R; \Fp)}(\Ga, \Ga \tensor_R M) \]
with $n$\textsuperscript{th} graded piece given by $R[F] \tensor_R M$ for $n = 0$ and $R[F]/F^{v_p(n) + 1} \tensor_R M [-2n]$ for $n > 0$.
\end{lemma}
\begin{proof}
For $R = M = \Fp$ this is \cite[Théorème 1.3]{breenext}.
This immediately implies the case of arbitrary $M$, as by the Breen--Deligne resolution we may commute out the relevant filtered colimits.
The case of arbitrary $R$ follows immediately, as in \Cref{WbigperftoA1perf} and \Cref{QRHomGaGa}.
\end{proof}

\begin{remark}
\label{rmk:stringstacks}
We expect that the failure in \Cref{derivedcounterexample} comes from using the ``wrong'' definition of ring stack.
The correct definition should be obtained by animating, in a suitable manner, the functor sending a polynomial $R$-algebra $S$ to the free ``strict rig scheme'' on $S$, given by
\[ \bigsqcup_{m, n} \Spec \Gamma_R^m \Gamma_R^n(S) \mpunc. \]
Full faithfulness should then hold for the obtained theory of ``strict ring stacks''.
\end{remark}

The notion of strict ring stack above is motivated by the definition of a presheaf with transfers along multiplicative polynomial laws \cite{robypls,robympls,mpls} (see also \cite{akhilmpls}).
This is a replacement for the usual notion of abelian presheaf which in particular kills the problematic $\Ext$-groups of \cite{breenext}.

\subsection{Outlook}

\begin{remark}[Shtukas with multiple legs]
\Cref{ZpSynmainthm} suggests a way to define ``$p$-adic shtukas with multiple legs'' over $\Spf \Zp$.
Namely, for any $n$ we can consider the $n$-fold power of $\ZpSyn$ over the stack of monoid stacks, and declare sheaves on this stack to be shtukas with $n+1$ legs.
Assuming \Cref{conj:SSyn}, this construction should be compatible with that of \cite{berkeleylectures}, as the induced formal group on $\Ainf[\tfrac{1}{\mu}]$ is a cyclotomic twist of $\Gmhat$, with section given by $1 + \xi$.
This construction would therefore produce a natural deperfected integral variant of that of \emph{op.\ cit.}\footnotemark
\footnotetext{%
  Beware that in this imperfect setting one cannot pass freely between all untilts of a given perfectoid ring of characteristic $p$, as the section of the formal group remembers some information about where an untilt lives in $\Ainf$.
  Inverting the Frobenius forgets this information.
}
\end{remark}

\begin{remark}[The function field case]
Given a smooth affine curve $X$ over $\Fq$, one can define a stack $X^\shtuka$, the ``shtukification'' of $X$, as
\[ X^\shtuka \defeq \coeq \prn[\Big]{X \times X \setminus \Delta \xrightrightarrows[\phi_q \times \id]{\id \times \id} X \times X} \mpunc[-1.1ex], \]
where $\phi_q$ is the $q$-power Frobenius.\footnotemark
\footnotetext{A more sophisticated construction would incorporate some sort of filtration data along $\Delta$, as in the syntomification.}
A sheaf on this stack is essentially a shtuka with a ``universal'' leg on $X$.

We expect that $X^\shtuka$ admits a similar moduli description to the stacks studied in this paper, but with usual ring stacks replaced with some sort of ``$X$-strict ring stacks'', which should be closely related to Faltings's ``strict $\calO$-modules'' \cite{faltingsstrictmodules}.
\end{remark}

We will now state some conjectures related to \Cref{ZpSynmainthm}, but we first need to recall some results of \cite{evenstack}.

\begin{recollection}[The prismatization of $\bbS$]
\label{evenstacks}
Following \cite{evenstack}, we define the stack $\bbS^\Nyghat$ to be the moduli stack of ($1$-dimensional, commutative) formal groups with a section (which identifies with the universal formal group $\Guniv$), and we define
\[ \Mm[\bbS]^\Nyghat \defeq \colim_{\Delta^\opp} \prn*{\Spec \pi_{2*} \TCm(\Mm[\MU^{\tensor (\bullet + 1)}])} / \Gm \mpunc. \]
As the notation suggests, these objects are completions of a stack $\bbS^\Nyg$ and a monoid stack $\Mm[\bbS]^\Nyg \to \bbS^\Nyg$ defined in \emph{op.\ cit.\@} (see also \cite[\S II.7]{sanaththesis}). 
As in the usual prismatic story, this descends to a monoid stack $\Mm[\bbS]^\Syn \to \bbS^\Syn$, and there is in fact a map $\ZpSyn \to \bbS^\Syn$ which pulls $\Mm[\bbS]^\Syn$ back to $\Mm^\Syn$.
There is also a map from $\bbS^\Syn$ to the moduli stack of formal groups $\Mfg$, and the composite map $\ZpSyn \to \Mfg$ classifies the Drinfeld formal group of \cite[Construction 4.1]{devenfg}.
\end{recollection}

\begin{remark}
Some of the conjectures below make use of the theory of analytic stacks of \cite{analyticstacks}.
We regard ordinary rings (potentially with an adic topology) as analytic rings via $R \mapsto (R, R)_\blacksquare$, and similarly regard suitably geometric stacks as analytic stacks.

We expect that the strict ring stacks of \Cref{rmk:stringstacks} can be formulated in this setting as well; note that it is necessary to do so, as the stacks of \cite{analyticstacks} are always derived.
\end{remark}

\begin{conjecture}[Refined $\TCm$ of $\bbQ$]
\label{conj:TCmref}
Regard $\bbS^\Nyghat$ and $\Mm[\bbS]^\Nyghat$ as analytic stacks (in the sense of \cite{analyticstacks}) in the natural way.
Then the graded cohomology ring of the analytic stack of $\bbQ$-algebra structures on $\Mm[\bbS]^\Nyghat$ is given by
\[ \lim_\Delta \pi_{2*} \prn*{\TCmref(\bbQ) \tensor \MU^{\tensor (\bullet + 1)}} \mpunc, \]
where $\TCmref$ denotes Efimov's refined $\TCm$ \cite{refinedTCm}.\footnotemark
\footnotetext{%
  It seems plausible that, for the conjecture to hold, one needs to modify the definition of $\Mm[\bbS]^\Nyghat$ slightly to make it more analytically sensible.
  It is also not clear to the authors whether one needs to ask for a strict ring structure, or whether this is instead obtained automatically from $\bbQ$-linearity.
}
\end{conjecture}

\begin{conjecture}
\label{conj:ZpSynSSyn}
The map $\ZpSyn \to (\bbS^\Syn)_{\Spf \Zp}$ identifies the source with the stack of ring structures on the monoid stack $(\Mm[\bbS]^\Syn)_{\Spf \Zp}$.
\end{conjecture}

In light of \Cref{ZpSynmainthm}, \Cref{conj:ZpSynSSyn} would follow from the following conjecture (which should be quite approachable) together with the claim that Raksit's height $\geq 1$ deformed filtered de Rham complexes (see \cite{deformeddR}) are not defined functorially for rings.

\begin{conjecture}
\label{conj:SSyn}
The map from $\bbS^\Syn$ to the stack of abelian monoid stacks classifying $\Mm[\bbS]^\Syn$ is fully faithful.
\end{conjecture}

The maximal subgroup of $\Mm[\bbS]^\Syn$ is very tractable, so the main difficulty in the conjecture above is extending group homomorphisms to monoid homomorphisms.

\begin{conjecture}[The functor of points of $\bbF_1^\Syn$]
\label{conj:f1syn}
Let $\bbS^\Prismhat$ denote the complement of the zero section in $\Guniv = \bbS^\Nyghat$ (an analytic stack).
Then, for a $p$-nilpotent ring $R$, an $R$-point of the stack $\widehat{\bbF}_1^\Syn$ of \cite{f1syn} is given by a formal group over $R$ together with a $p$-complete strict ring structure on the pullback of the monoid stack $\Mm[\bbS]^\Nyghat$ to $\Spec R \times_{\Mfg} \bbS^\Prismhat$.
\end{conjecture}

Assuming (analytic versions of) \Cref{conj:ZpSynSSyn,conj:SSyn}, \Cref{conj:f1syn} should be equivalent to the following conjecture purely about formal groups and Cartier--Witt divisors.

\begin{conjecture}
\label{conj:f1syn2}
For a $p$-nilpotent ring $R$, an $R$-point of the stack $\widehat{\bbF}_1^\Syn$ of \cite{f1syn} is given by a formal group over $R$ and a Cartier--Witt divisor on the pullback $\Spec R \times_{\Mfg} \bbS^\Prismhat$, along with an isomorphism over this pullback between the given formal group and the Drinfeld formal group compatible with the natural sections of each.\footnotemark
\footnotetext{%
  Here we mean the section $\tilde{\hspace{0.1ex}s}$ of \cite[\S 2.10.7]{drinfeldfg}, which lands in the Drinfeld formal group (rather than just its algebraization) when restricted to the appropriate analytic localization $\Zp^\Prismhat$ of $\Zp^\Nyghat$.
  We are also assuming the existence of a reasonable theory of Cartier--Witt divisors on the relevant stacks.
}
\end{conjecture}

\appendix
\addtocontents{toc}{\protect\setcounter{tocdepth}{1}}

\section{Pseudodescendable sheaves and divided powers}
\label{sec:LGamma}

The following definition is inspired by the descent condition used in Clausen--Scholze's theory of analytic stacks \cite[Lecture 19, 00:01:20]{analyticstacks},\footnotemark{} and of course by \cite{descendability}.
\footnotetext{%
  Note that this definition was later revised, as it (presumably) does not yield a topos; see \cite[Definition 4.2.1]{FFandR} and \cite[Remark 1.22]{gestalten}.
}

\begin{definition}
\label{pdescobj}
Let $A^\bullet$ be a cosimplicial $\bbE_0$-algebra in a presentably $\bbE_0$-monoidal stable category $\calC$, with unit $A^{-1}$.
Then we say that $A^\bullet$ is \emph{pseudodescendable} if the natural map
\[ A^{-1} \to \flim_n \Tot^n A^\bullet \]
is an equivalence in $\Pro(\calC)$.
Note that this notion is preserved under exact $\bbE_0$-monoidal functors.

We say that an augmented cosimplicial $\bbE_1$-ring $A^\bullet$ in a presentably monoidal stable category $\calC$ is (left) pseudodescendable if the induced cosimplicial $\bbE_0$-algebra in $\RMod_{A^{-1}}(\calC)$ is.
\end{definition}

\begin{remark}
\label{rmk:pdescdesc}
Note that a map of $\Einfty$-rings in a presentably symmetric monoidal stable category is descendable if and only if its Čech conerve is pseudodescendable.

Note also that if $A^\bullet$ is a pseudodescendable augmented cosimplicial $\Einfty$-ring in some presentably symmetric monoidal stable category, then $A^{-1} \to A^0$ is descendable.
Indeed, each term of the cosimplicial diagram $A^\bullet$ is an $A^0$-module, so, as $A^{-1}$ is a retract of some finite stage of the $\Tot$ tower, it lies in the thick $\tensor$-ideal generated by $A^0$.
\end{remark}

\begin{remark}
\Cref{pdescobj} is somewhat too broad for many purposes.
For example, although one can easily show that $\LMod_{(-)}$ satisfies ineffective descent along pseudodescendable augmented cosimplicial $\bbE_1$-rings, this descent is not effective.

To see this, consider the augmented cosimplicial ring $A^\bullet$ obtained by applying $\RGamma$ to the Čech nerve of the cover $\bbA^1 \sqcup \bbA^1 \to \bbP^1$ over a field.
Then the object
\[ \cofib(A^{-1} \to \Tot^n A^\bullet) \]
is equivalent for each $n$ to a vector space concentrated in cohomological degree $n$, so the tower
\[ \cofib(A^{-1} \to \flim_n \Tot^n A^\bullet) \]
is pro-zero.
Thus $A^\bullet$ is pseudodescendable, even though it clearly does not have descent for modules.
\end{remark}


\begin{definition}
\label{pdesctop}
Given an animated ring $R$, we define the category $\Sh(R_\pdesc)$ of \emph{pseudodescendable sheaves} over $R$ as the localization of $\PSh(\aCAlg_R^\opp)$ generated by the following morphisms:
\begin{enumerate}
\item
  $i(S) \union i(T) \to i(S \times T)$, where $i$ is the Yoneda embedding;
\item
  $\colim_{\Delta^\opp} i(S^\bullet) \to i(S^{-1})$, where $S^\bullet$ is a pseudodescendable augmented cosimplicial animated $R$-algebra.
\end{enumerate}
\end{definition}

\begin{warning}
The category $\Sh(R_\pdesc)$ is presumably not a topos, but this does not matter for our purposes.
\end{warning}

\begin{remark}
\label{rmk:affstkpdesc}
Note that if $A$ is a derived $R$-algebra then $\Spec A$ is a pseudodescendable sheaf.
\end{remark}

\begin{lemma}
\label{excisivetot}
Let $F : \calC \to \calD$ be a $k$-excisive functor between stable categories, and let $X^\bullet \in \Fun(\Delta, \calC)$ be a cosimplicial object of $\calC$. Then the natural map
\[ \flim_n F(\Tot^n X^\bullet) \to \flim_n \Tot^n F(X^\bullet) \]
is an equivalence in $\Pro(\calD)$.
\end{lemma}
\begin{proof}
We have a natural factorization
\[ \flim_n F(\Tot^n X^\bullet) \to \flim_n \Tot^{n k} F(\sk^n X^\bullet) \xrightarrow{f} \flim_n \Tot^n F(X^\bullet) \mpunc, \]
where the first map is an equivalence by \cite[Proposition 2.10]{Kthypoly} and \cite[Proposition 3.37]{bmdefthy}.
It therefore suffices to see that the fiber of $f$ is pro-zero.

Fix some $n \in \bbN$ and consider the map
\[ g : \sk^{n k} F(X^\bullet) \to \sk^{n k} F(\sk^n X^\bullet) \mpunc. \]
We claim that the diagram
\[\xymatrix{
  \sk^{n k^2} F(\sk^{n k} X^\bullet) \ar[rr]^{f_{n k}} \ar[d] && \sk^{n k} F(X^\bullet) \ar[lld]_{g} \ar[d] \\
  \sk^{n k} F(\sk^n X^\bullet) \ar[rr]^{f_n} && \sk^n F(X^\bullet) \\
}\]
commutes.
The top triangle commutes because the left vertical map is by definition given by the composite
\[ \sk^{n k^2} F(\sk^{n k} X^\bullet) \to \sk^{n k} F(\sk^{n k} X^\bullet) \to \sk^{n k} F(\sk^n X^\bullet) \mpunc. \]
The bottom triangle commutes because maps to $\sk^n F(X^\bullet)$ are determined by their restrictions to degree $\leq n$, where both maps in question are the identity.
Finally, the commutative square obtained by patching the triangles agrees with the standard one for the same reason.

Taking limits and rearranging, we obtain a diagram
\[\xymatrix{
  \Tot^{n k^2} F(\sk^{n k} X^\bullet) \ar[rr]^{f_{n k}} \ar[d]^{f_{n k}} && \Tot^{n k} F(X^\bullet) \ar[d]^{\id} \\
  \Tot^{n k} F(X^\bullet) \ar[rr]^{\id} \ar[d]^g && \Tot^{n k} F(X^\bullet) \ar[d] \\
  \Tot^{n k} F(\sk^n X^\bullet) \ar[rr]^{f_n} && \Tot^n F(X^\bullet)
  \mpunc,
}\]
so the map on fibers
\[ \fib\prn[\big]{\Tot^{n k^2} F(\sk^{n k} X^\bullet) \xrightarrow{f_{n k}} \Tot^{n k} F(X^\bullet)} \to \fib\prn[\big]{\Tot^{n k} F(\sk^n X^\bullet) \xrightarrow{f_n} \Tot^n F(X^\bullet)} \]
factors through $0$.
As this holds for all $n \in \bbN$, we conclude.
\end{proof}

\begin{lemma}
\label{LGammalax}
For a ring $R$, the functor
\[ \LGamma_R^k : \calD(R)^{\leq 0} \to \calD(R) \]
admits a lax symmetric monoidal structure.
\end{lemma}
\begin{proof}
Recall that $\LGamma_R^k$ is defined by left Kan extension of the lax symmetric monoidal functor $\Gamma_R^k$ along the inclusion $\Mod_R^\proj \subseteq \calD(R)^{\leq 0}$.
Because $\calD(R)^{\leq 0}$ is the free sifted cocompletion of $\Mod_R^\proj$, we can apply \cite[Proposition 4.8.1.10(4)]{ha} to obtain a lax symmetric monoidal structure on $\LGamma_R^k$ as well.
\end{proof}

\begin{corollary}
\label{LGammaMod}
Let $R$ be a ring, and let $S$ be an $\Einfty$-$R$-algebra.
Then the functor
\[ \LGamma_R^k : \calD(R) \to \calD(R) \]
lifts to a $k$-excisive functor
\[ \calD(S) \to \calD(\LGamma_R^k(S)) \mpunc, \]
where the ring structure on $\LGamma_R^k(S)$ is obtained from \Cref{LGammalax}.
\end{corollary}
\begin{proof}
It follows from \Cref{LGammalax} that $\LGamma_R^k$ lifts to a functor from $\calD(S)^{\leq 0}$ to $\calD(\LGamma_R^k(S))$.
This functor is still $k$-excisive, as the forgetful functors to $\calD(R)^{\leq 0}$ on the left and $\calD(R)$ on the right are conservative and preserve colimits.
Restricting this functor to $\Perf(S)^{\leq 0}$ and applying \cite[Theorem 3.36]{bmdefthy}, we obtain an extension to all of $\Perf(S)$, which lifts the usual $\LGamma_R^k$ by another application of \cite[Theorem 3.26]{bmdefthy}.
Finally, as $\LGamma_R^k$ commutes with filtered colimits, left Kan extension along $\Perf(S) \subseteq \calD(S)$ yields the desired functor.
\end{proof}

\begin{corollary}
\label{LGammadesc}
Let $R$ be a ring, and let $S \to T$ be a descendable map of $\Einfty$-$R$-algebras.
Then for each $k \in \bbN$, the augmented cosimplicial $\Einfty$-algebra
\[ \LGamma_R^k(S) \to \LGamma_R^k(T^{\tensor_S \bullet+1}) \mpunc, \]
with algebra structure obtained from \Cref{LGammalax}, is pseudodescendable.
In particular, by \Cref{rmk:pdescdesc}, the map
\[ \LGamma_R^k(S) \to \LGamma_R^k(T) \]
is descendable.
\end{corollary}
\begin{proof}
Because $S \to T$ is descendable, the natural map
\[ S \to \flim_n \Tot^n T^{\otimes_S \bullet+1} \]
is an equivalence in $\Pro(\calD(S))$, so by \Cref{LGammaMod} and \Cref{excisivetot} we find that
\[ \LGamma_R^k(S) \to \flim_n \Tot^n \LGamma_R^k(T^{\otimes_S \bullet+1}) \]
is an equivalence in $\Pro(\calD(\LGamma_R^k(S)))$.
By naturality of the map in \cite[Corollary 3.4.1.5]{ha}, the module structure obtained from \Cref{LGammaMod} agrees with that obtained from the $\Einfty$-structure, so we conclude.
\end{proof}

\begin{corollary}
\label{flatGammadesc}
Let $R$ be a ring, and let $S \to T$ be a flat descendable map of flat $R$-algebras.
Then for every $k \in \bbN$, the natural map
\[ \colim_{\Delta^\opp} \Spec \Gamma_R^k(T^{\otimes_S \bullet+1}) \to \Spec \Gamma_R^k(S) \]
is an equivalence in $\Sh(R_\pdesc)$.
\end{corollary}
\begin{proof}
By \Cref{LGammadesc} we have an equivalence
\[ \LGamma_R^k(S) \bij \flim_n \Tot^n \LGamma_R^k(T^{\otimes_S \bullet+1}) \]
in $\Pro(\calD(\LGamma_R^k(S)))$.
Since everything involved is a flat $R$-algebra, we are free to replace $\LGamma_R^k$ with $\Gamma_R^k$, so the map in question is an equivalence by the definition of $\Sh(R_\pdesc)$.
\end{proof}

\begin{recollection}
\label{abmon}
Recall that the category $\Mod_{\bbN}(\calC)$ of abelian monoids in a category $\calC$ with finite products is defined as $\PSigma(\Lat_{\bbN}; \calC)$, where $\Lat_{\bbN}$ is the full subcategory of $\CMon(\Set)$ on those monoids isomorphic to $\bbN^n$ and $\PSigma$ is the category of product-preserving presheaves (see e.g.\ \cite[\S 7]{marckuenneth}).
\end{recollection}

\begin{construction}
\label{Ntrdef}
Let $R$ be a ring.
Consider the functor
\begin{align*}
  \Ntr(-) : (\CAlg_R^\mathrm{flat})^\opp &\to \Mod_{\bbN}(\Sh(R_\pdesc)) \\
  S &\mapsto \bigsqcup_n \Spec \Gamma^n(S)
\end{align*}
on the category of flat $R$-algebras sending an algebra $S$ to the free commutative monoid ind-scheme on $\Spec S$.
It follows from \Cref{flatGammadesc} together with the fact that $\Ntr$ preserves coproducts that it is a cosheaf for the flat descendable topology, i.e.\ the topology generated by flat descendable morphisms.
We will also consider the right Kan extension of $\Ntr$ to the category of flat descendable sheaves on $(\CAlg_R^\mathrm{flat})^\opp$, which we will denote using the same notation.
\end{construction}

\begin{remark}
\label{rmk:Ntrlaxmonoidal}
We construct a lax symmetric monoidal structure on $\Ntr$ as follows.
We have an obvious lax monoidal structure on the free $\Einfty$-monoid functor on prestacks; writing $F$ for this functor, we in particular have a bilinear natural transformation $F(X) \times F(Y) \to F(X \times Y)$ for any flat affine schemes $X$ and $Y$.
Taking $\Spec H^0(-)$ (separately in each weight for the natural grading) and using the fact that $\Gamma^m(M) \tensor \Gamma^n(N) \isom (M^{\tensor m} \tensor N^{\tensor n})^{\Sigma_m \times \Sigma_n}$ for flat modules $M$ and $N$ (so $\Spec H^0(-)$ is symmetric monoidal on the relevant subcategory), we obtain the necessary bilinear natural transformation $\Ntr(X) \times \Ntr(Y) \to \Ntr(X \times Y)$ for $X$ and $Y$ representable.\todo*[better way of doing this?]
We then conclude by applying \cite[Proposition 4.8.1.10]{ha} and the fact that the inclusion of sheaves into presheaves commutes with products.
\end{remark}

\begin{construction}
\label{Ntrptddef}
In the situation of \Cref{Ntrdef}, we also obtain a functor $\Ntr(-, -)$ on the category of pointed flat descendable sheaves on $(\CAlg_R^\mathrm{flat})^\opp$ given by the pushout
\[\xymatrix{
  \Ntr(*) \ar[d] \ar[r] & {*} \ar[d] \\
  \Ntr(X) \ar[r] & \pullbackcorner[ul][-1.8pc] \Ntr(X, *)
  \mpunc.
}\]
Note in the above that $\Ntr(*) \isom \bbN$.
Note also that $\Ntr(-, -)$ is lax symmetric monoidal, since it can be rewritten as $\Ntr(-) \tensor_{\bbN[\bbN]} \bbN$.
\end{construction}

\begin{remark}
\label{rmk:Ntrmonoid}
For a flat commutative affine monoid $R$-scheme $M$, there is by construction a natural monoid map $\Ntr(M, e) \to M$ in $\Mod_{\bbN}(\Sh(R_\pdesc))$, where $e$ is the unit.
If $M$ is a rig scheme, then this is in fact a rig map, as follows from the construction of the rig structure via \Cref{rmk:Ntrlaxmonoidal}.
\end{remark}

\begin{remark}
\label{rmk:Wratplus}
It follows from \Cref{pushoutcolim} that, given a pointed flat affine $R$-scheme $(X, *)$, the stack $\Ntr(X, *)$ identifies naturally with
\[ \colim_n \mkern2mu \Sym^n X \mpunc, \]
where $\Sym^n$ is the symmetric power scheme and the transition maps are induced by $* \to X$.

In particular, the stack $\Ntr(\Mm, 0)$ identifies with the rig of \emph{positive rational Witt vectors} $\Wratplus$, which sends a ring $R$ to the submonoid $1 + t R[t] \subseteq 1 + t R \Brk{t} = \Wbig(R)$ of big Witt vectors with finitely many nonzero terms (see \cite[Construction A.4]{akhilrh}).
\end{remark}

\begin{lemma}
\label{pushoutcolim}
Let $M_* \in \Mod_\bbN(\Fun(\bbN, \Ani))$ be an abelian monoid in $\bbN$-graded anima and $\bbN \to M_*$ a homomorphism classified by an element $x \in M_1$.
Then the underlying anima of the pushout $M_*/x$ of $\coprod_i M_i$ along $\bbN \to *$ is naturally equivalent to $\colim_i M_i$, where the colimit is taken along $(-) + x$.
\todo*[would be nice to shrink this]
\end{lemma}
\begin{proof}
Recall that $\colim_i M_i$ can be computed by the coequalizer
\[\xymatrix{
  \coprod_i M_i \ar@<-.4ex>[r]_{+x} \ar@<.4ex>[r]^{\id} & \coprod_i M_i \mpunc.
}\]
To construct a map from $\colim_i M_i$ to $M_*/x$, it suffices to produce a map from this diagram to the bar construction whose colimit computes $M_*/x$.
The diagram
\[\xymatrix{
  \coprod_i M_i \ar@<-.4ex>[rrr]_{+x} \ar@<.4ex>[rrr]^{\id} \ar[d]^{\id \times \{1\}} &&& \coprod_i M_i \ar[d]^{\id} \\
  \coprod_i M_i \times \bbN \ar@<-.4ex>[rrr]_{(y, n) \mapsto y + n x} \ar@<.4ex>[rrr]^{(y, n) \mapsto y} &&& \coprod_i M_i \\
}\]
gives a map from the above coequalizer diagram to the first two terms of the semisimplicial bar construction computing $M_*/x$.
Composing this with the inclusion of the bottom coequalizer into the full semisimplicial bar construction and taking colimits, we obtain a natural map $\colim_i M_i \to M_*/x$.

Since both sides commute with sifted colimits, we may now assume $M_*$ is ``freely pointed'', i.e.\ of the form $M'_* \times \bbN$ with $x = (0, 1)$.
In this case, the pushout $M_*/x$ is given by $\coprod_i M'_i$, while the colimit is computed by $\colim_i \coprod_{j=0}^i M'_j$, and the map constructed earlier is the usual equivalence
\[ \colim_i \coprod_{j=0}^i M'_j \bij \coprod_i M'_i \mpunc. \]
\\[-2.1em]
\end{proof}

\section{Affine stacks}
\label{sec:affstk}

\subsection{Generalities}

\begin{recollection}[{\cite[\S 4.2]{arpondR}}]
A \emph{derived algebraic context} consists of a presentably symmetric monoidal stable category $\calC$ with a compatible t-structure together with a small full subcategory $\calC^0 \subseteq \calC^\heart$, satisfying a certain set of conditions described in \cite[Definition 4.2.1]{arpondR}.\footnotemark
\footnotetext{We will generally elide the subcategory $\calC^0$ in the notation.}

Given a derived algebraic context $\calC$, one defines a category $\DAlg(\calC)$ of \emph{derived (commutative) rings} refining the usual category of $\Einfty$-rings in $\calC$.
The construction $\calC \mapsto \DAlg(\calC)$ is functorial in derived algebraic contexts \cite[Remark 4.2.25]{arpondR}.
\end{recollection}

\begin{lemma}
\label{DAlginternalvsexternal}
Let $\calC$ be a presentably symmetric monoidal stable category with compatible t-structure, and let $A \to B$ be a map of commutative rings in $\calC^\heart$ such that
\[ (\Mod_A(\calC), (\Mod_A(\calC)^{\leq 0})^{\omega,\proj}) \to (\Mod_B(\calC), (\Mod_B(\calC)^{\leq 0})^{\omega,\proj}) \]
is a morphism of derived algebraic contexts.\footnotemark
\footnotetext{Here $(-)^{\omega,\proj}$ denotes the full subcategory of compact projective objects.}
Then the natural map
\[ \DAlg(\Mod_B(\calC)) \to \DAlg(\Mod_A(\calC))_{B/} \]
obtained from the right adjoint to the base change functor is an equivalence.
\end{lemma}
\begin{proof}
Note that we have a natural equivalence $\Mod_B(\calC) \isom \Mod_B(\Mod_A(\calC))$ of derived algebraic contexts, so we may as well replace $\calC$ by $\Mod_A(\calC)$.

Now note that the source of the functor in question is monadic over $\calC$ by definition, and the target is monadic over $\calC$ by Beck--Lurie monadicity \cite[Corollary 4.7.3.5]{ha} (see \cite[Notation 4.2.28(c)]{arpondR}).
Since this functor commutes with the forgetful functors to $\calC$, to see that it is an equivalence it suffices to verify that it also commutes with the free functors.
This amounts to the claim that the natural map
\[ \LSym_B(B \otimes_A M) \bij B \otimes \LSym_A(M) \]
is an equivalence for all $M \in \calC$, which is contained in \cite[Remark 4.2.25]{arpondR}.
%
\end{proof}

\begin{lemma}
\label{DAlgdescent}
Let $\calC^\bullet$ be a cosimplicial derived algebraic context satisfying the (opposite) conditions of \cite[Corollary 4.7.5.3]{ha} (including conservativity) on underlying categories.
Then $\DAlg(\calC^\bullet)$ is a limit diagram.
\end{lemma}
\begin{proof}
By \cite[Remark 4.2.25]{arpondR}, the conditions of \emph{loc.\ cit.\@} for $\DAlg(\calC^\bullet)$ follow from those for $\calC^\bullet$, so we conclude.
\end{proof}

\begin{definition}
Given a (possibly derived) c-stack $X$, we define the category $\DAlg(X)$ by right Kan extension from affines.
This category comes with a natural forgetful functor $\DAlg(X) \to \CAlg(X)$ which admits both left and right adjoints.
\end{definition}

\begin{remark}
\label{gradedDAlgwelldefined}
By virtue of \Cref{DAlginternalvsexternal} and \Cref{DAlgdescent}, we see that the category of graded derived rings of \cite[Construction 4.3.4]{arpondR} agrees with $\DAlg(\BGm)$, and similarly for $\BGmperf$, $\BMm$, $\BMmperf$, etc.\footnotemark
\footnotetext{The equivalences between the derived categories of these c-stacks and the relevant categories of graded modules follow from the argument of \cite[Theorem 4.1]{tasosa1modgm}.}
\end{remark}

\begin{remark}
\label{gradedDAlgrightadjoint}
It is easy to see that the natural functor from $\bbN$-graded derived rings to $\bbN[\tfrac{1}{p}]$-graded derived rings is fully faithful and commutes with all limits and colimits.
In particular, it has a right adjoint, given by restriction of the grading.
\end{remark}

\begin{definition}[{\cite[Remark 4.5]{mmaffstk}}]
We say that a map of derived c-stacks $f : Y \to X$ is a \emph{relative affine stack} if its pullback to any affine derived scheme over $X$ is an affine stack in the sense of \cite[Definition 3.8]{mmaffstk}.
We may also say \emph{relatively affine} etc.\ if the meaning is clear from context.
Note that if $X$ and $Y$ are classical and the formation of $f_* \calO_Y$ is compatible with base change, then \cite[Corollary 3.6]{mmaffstk} shows that $Y$ is determined by $f_* \calO_Y$ as an object of $\DAlg(X)$.
\end{definition}

\begin{remark}
It follows from \cite[Proposition 4.2]{mmaffstk} that relative affineness may be checked after an fpqc cover.
\end{remark}

\subsection{Affineness of the prismatization}

The purpose of this section is to prove the following statement.

\begin{theorem}[Bhatt--Lurie]
\label{synrelaffine}
The syntomification of a $p$-complete animated ring is relatively affine over $\ZpSyn$.
\end{theorem}

We will deduce the \namecref{synrelaffine} from the following result.

\begin{proposition}
\label{GaSynaffineness}
The stack $\Ga^\Syn \to \ZpSyn$ is relatively affine.
\end{proposition}
\begin{proof}
It suffices to check this after pulling back to the affine formal scheme $\ZpcycNyg \times_{\BGm} *$.
After this pullback $\Ga^\Syn$ is given as $W/M$ for a filtered Cartier--Witt divisor $M \to W$ described in \cite[Example 5.5.6]{fgauges}.
Note that the pushout yielding $M$ is in fact defined over the regular noetherian ring $\bbZ_{(p)}[u]$, so by \cite[Proposition 4.11]{mmaffstk} and unipotence of $M$ (which follows from unipotence of $W$ and $\Gasharp$) the stack $\bfB M$ is affine.
Thus by \cite[Corollary 4.4]{mmaffstk} we find that $W/M$ is affine as well.
\end{proof}

\begin{proof}[Proof of \Cref{synrelaffine}]
In light of \Cref{GaSynaffineness}, this follows from \cite[Remark 3.9]{mmaffstk} and the fact that the transmutation of a ring is contained in the category generated under limits by that of $\Ga$.
\end{proof}

\section{Base change}

Here we record a base change result for algebraic stacks with affine diagonal over formal affine schemes.
For the proof we refer to \cite[Proposition A.36]{prismaticdelta}, which in turn essentially follows \cite[Proposition 5.5.5]{dag}; see also \cite[\S 3.2]{bzfn}.

\begin{proposition}
\label{basechange}
Let $R$ be a ring complete with respect to a finitely-generated ideal $I$ and of bounded $I$-torsion, and let $R \to S$ be a morphism such that $S$ is of bounded $I$-torsion and is the $I$-completion of an $R$-module of finite $\Tor$-amplitude.
Let $f : X \to \Spf R$ be a c-stack with affine diagonal admitting a faithfully flat morphism from an $I$-completely flat affine formal $R$-scheme, and write $X'$ for the pullback
\[\xymatrix{
  X' \ar[d]^{f'} \ar[r]^{g'} \pullbackcorner & X \ar[d]^f \\
  \Spf S \ar[r]^g & \Spf R
  \mpunc.
}\]
Then for $M \in \calD(X)$ eventually-coconnective the base change map
\[ g^* f_* M \to f'_* g'^* M \]
is an equivalence.
Further, for $M \in \calD(X)$ eventually-coconnective and $N \in \calD(\Spf R)$ a filtered colimit of uniformly-eventually-coconnective perfect complexes,\todo*[change when new version of \cite{prismaticdelta} is available] the natural map
\[ f_* M \tensor N \to f_* (M \tensor f^* N) \]
is an equivalence; i.e.\ the projection formula holds for such $N$.
\end{proposition}
\begin{proof}
This is what the proof of \cite[Proposition A.36]{prismaticdelta} shows, although the result there is given only for a special case.
\end{proof}

%

\section{Transmutation}

We give here some general results on transmutation.
These are independent of the specific sort of algebraic geometry under consideration, so ``commutative ring'' here can be taken to refer to the classical notion or the animated version, and similarly for ``stack''.

\begin{lemma}
\label{accessiblerightadjoint}
If $\calC$ and $\calD$ are coaccessible categories, then a functor $F : \PShAcc(\calD) \to \calC$ admits a right adjoint if and only if it preserves small colimits and its restriction to $\calD$ is accessible.
\end{lemma}
\begin{proof}
By passing to opposite categories in \citekerodon{02FV}, it suffices to show that for every $c \in \calC$, the composite
\[ \PShAcc(\calD)^\opp \xrightarrow{F^\opp} \calC^\opp \xrightarrow{h_c} \Ani \]
is representable.
Since $F^\opp$ and $h_c$ preserve limits, it suffices by the definition of $\PShAcc(\calD)$ to verify that the induced functor
\[ \calD^\opp \to \Ani \]
is accessible.
This is the composition of $F^\opp|_{\calD^\opp}$, assumed to be accessible, and $h_c$, which is accessible because it is representable and $\calC$ is coaccessible.
\end{proof}

\begin{proposition}
\label{translim}
Let $S$ be a prestack, and write $\PStk_{/S}$ for the category of prestacks over $S$.
A functor $F : \CAlg_A^\opp \to \PStk_{/S}$ arises by transmutation if and only if it preserves small limits.
\end{proposition}
\begin{proof}
Given such a functor $F$, the opposite functor $F^\opp : \CAlg_A \to \PStk_{/S}^\opp$ preserves colimits.
By presentability of $\CAlg_A$, $F^\opp$ therefore has a right adjoint, so $F$ has a left adjoint.
Applying \Cref{accessiblerightadjoint} and \cite[Proposition A.13]{dirac2}, we get a chain of equivalences
\[ \FunR(\CAlg_A^\opp, \PStk_{/S})^\opp \equiv \FunL(\PStk_{/S}, \CAlg_A^\opp) \equiv \FunAcc(\CAlg_{/S}, \CAlg_A) \]
from limit-preserving functors $\CAlg_A^\opp \to \PStk_{/S}$ to $A$-algebra prestacks over $S$.
Unwinding definitions, the map from right to left is given by transmutation.
\end{proof}

\begin{lemma}
\label{transff}
Let $A$ be a commutative ring, and let $\calR$ be an $A$-algebra stack over a c-stack $X$ inducing a fully faithful map from $X$ to the stack of $A$-algebra stacks.
Then for any $A$-algebra $B$, the natural map from the transmutation $B^\calR$ to the stack of $B$-algebra stacks is fully faithful as well.
\end{lemma}
\begin{proof}
Exercise.
\end{proof}

\begin{lemma}
\label{ringstksheaf}
For a ring $A$, the functor sending a ring $R$ to the category of $A$-algebra stacks (for some finitary topology) is a sheaf (for the same topology).
\end{lemma}
\begin{proof}
It is clear that this functor preserves products, so it remains to see descent along single-element covers.
For this, we use \cite[Corollary 4.7.5.3]{ha}.
The first condition of \emph{loc.\ cit.\@} is clear.
The second follows from the fact that, for maps $S \xleftarrow{f} R \xrightarrow{g} T$, the natural map $g^* f_* \calR \to (f \tensor_R T)_* (g \tensor_R S)^* \calR$ is an equivalence for any $A$-algebra stack $\calR$ over $T$, as both are given by
\[ S' \mapsto \calR(S' \tensor_R T) \mpunc. \qedhere \]
\end{proof}

\section{Maps of algebras}
\label{sec:algmaps}

\begin{lemma}
\label{funlim}
Let $\calC$ and $\calD$ be categories, with $\calC$ small, and let $F, G \in \Fun(\calC, \calD)$ be given. Then the anima
\[ \Map_{\Fun(\calC, \calD)}(F, G) \]
can, functorially in $F$, be written as a limit of the anima
\[ \Map_\calD(F(c), G(c')) \]
for $c, c' \in C$.
\end{lemma}
\begin{proof}
Recall that the category $\Delta^1$ generates $\Cat$ under colimits (see e.g.\ \cite[Example 2.7]{monadictower}).
As $\Fun(-, \calD)$ sends colimits to limits, and mapping anima in a limit of categories are computed by a limit, it suffices to prove the claim for $\Delta^1$.
In this case, the desired anima is given by the pullback of the natural diagram
\[\begin{gathered}[b]\xymatrix{
  & \Map_\calD(F(0), G(0)) \ar[d] \\
  \Map_\calD(F(1), G(1)) \ar[r] & \Map_\calD(F(0), G(1))
  \mpunc.
}\\[-1.5\dp\strutbox]\end{gathered}
\qedhere
\]
\end{proof}

\begin{lemma}
\label{mapobjcalg}
Let $\calC$ be a symmetric monoidal category, and suppose we have $X, Y, Z \in \CAlg(\calC)$.
If we are given a map $X \to Y$ in $\CAlg(\calC)$ such that, for each $n$, the induced map
\[ \Map_\calC(Y^{\tensor n}, Z) \to \Map_\calC(X^{\tensor n}, Z) \]
is an equivalence, resp. injective, then the induced map
\[ \Map_{\CAlg(\calC)}(Y, Z) \to \Map_{\CAlg(\calC)}(X, Z) \]
is an equivalence, resp. injective, as well.

Similarly, if we are given a map $Y \to Z$ in $\CAlg(\calC)$ such that, for each $n$, the induced map
\[ \Map_\calC(X^{\tensor n}, Y) \to \Map_\calC(X^{\tensor n}, Z) \]
is an equivalence, resp. injective, then the induced map
\[ \Map_{\CAlg(\calC)}(X, Y) \to \Map_{\CAlg(\calC)}(X, Z) \]
is an equivalence, resp. injective, as well.
\end{lemma}
\begin{proof}
The category $\CAlg(\calC)$ is, by definition, a full subcategory of $\Fun_\Finp(\Finp, \calC^\otimes)$.
Because the group of automorphisms of the identity functor of $\Finp$ is trivial, this embeds further as a full subcategory of $\Fun(\Finp, \calC^\otimes)$.
The functor corresponding to a commutative algebra $W \in \CAlg(\calC)$ sends $\langle n \rangle$ to the pair $(\langle n \rangle, (W, \ldots, W)) \in \calC^\otimes$.
By \Cref{funlim}, we therefore have that $\Map_{\CAlg(\calC)}(W, Z)$ can, functorially in $W$, be written as a limit of the anima
\[ \Map_{\calC^\otimes}((\langle n \rangle, (W, \ldots, W)), (\langle m \rangle, (Z, \ldots, Z))) \mpunc, \]
with indexing category independent of $W$.
This mapping anima in $\calC^\otimes$ is a coproduct of products of mapping anima of the form $\Map_\calC(W^{\otimes k}, Z)$.

It follows that the map
\[ \Map_{\CAlg(\calC)}(Y, Z) \to \Map_{\CAlg(\calC)}(X, Z) \]
can be written as a limit of coproducts of products of the induced maps
\[ \Map_\calC(Y^{\otimes k}, Z) \to \Map_\calC(X^{\otimes k}, Z) \mpunc. \]
As equivalences and injective maps of anima are preserved under limits and coproducts, the first statement follows.
The second statement follows similarly.
\end{proof}

\printbibliography

\setlength\parindent{0pt}

\end{document}